\documentclass{amsart}
\usepackage{latexsym,amsmath,amsthm,amssymb,fullpage}

\newtheorem{theorem}{Theorem}[section]
\newtheorem{lemma}[theorem]{Lemma}
\newtheorem{corollary}[theorem]{Corollary}
\newtheorem{proposition}[theorem]{Proposition}

\theoremstyle{remark}

\newtheorem{remark}[theorem]{Remark}

\theoremstyle{definition}
\newtheorem{definition}[theorem]{Definition}

\newcommand{\dR}{\ensuremath{\mathbb{R}}} 
\newcommand{\R}{\dR}

\newcommand{\1}{\mathbf 1} 
\newcommand{\ent}{\mathrm{Ent}} 
\newcommand{\var}{\mathbf{Var}} 

\begin{document}

\title{Mass transport and variants of the logarithmic\\  Sobolev inequality}
\author{Franck Barthe, Alexander V. Kolesnikov}
 \thanks{Second named author supported by RFBR 07-01-00536,  DFG Grant 436 RUS 113/343/0 and  GFEN 06-01-39003.}

\begin{abstract}
We develop the optimal transportation approach to modified log-Sobolev inequalities
and to isoperimetric inequalities. 
Various sufficient conditions for such inequalities are given. Some of them are new even in
the classical log-Sobolev case. The idea behind many of these conditions is that measures
with a non-convex potential may enjoy such functional inequalities provided they have
a strong integrability property that balances the lack of convexity.
In addition, several known criteria are recovered in a simple unified way by 
transportation methods and generalized to the Riemannian setting.
\end{abstract}
\maketitle

\setcounter{tocdepth}{1}
\tableofcontents
\section{Introduction}
This work deals with Sobolev inequalities and isoperimetric properties of absolutely 
continuous probability measures on Euclidean space or Riemannian manifolds. This subject is 
connected, among other fields, to analysis, probability theory, differential geometry, partial differential
equations. In particular such properties are crucial in the study of the concentration of measure phenomenon
and of  the regularizing effect and trend to equilibrium  of evolution equations.
 Several surveys were devoted to this quickly developing topic, see e.g.  \cite{ISLFr},  \cite{Bakry-rev}, \cite{bobkh97cbis},
\cite{guionnet-zegarlinski}, \cite{Gross-surv}, 
\cite{Led01}, \cite{ros01ip}.

\medskip

 Besides the Poincar\'e or spectral gap 
inequality, the logarithmic Sobolev inequality is the best studied Sobolev property for probability measures.
The basic example of a measure satisfying the logarithmic Sobolev inequality
\begin{equation}
\label{LSI0}
\mbox{\rm Ent}_{\mu} f^2 := \int f^2 \log \left(\frac{f^2}{\int f^2 \,d\mu}\right) \,d\mu
\le 2C \int |\nabla f|^2 \,d\mu
\end{equation}
is  the standard Gaussian measure  on $\R^d$, $d\gamma(x) = (2 \pi)^{-\frac{d}{2}} e^{-\frac{|x|^2}{2}} \,dx$
(with $C=1$). It is well known, that for every measure $\mu$ satisfying
(\ref{LSI0}) there exists $\varepsilon>0$ such that $e^{\varepsilon |x|^2} \in L^1(\mu)$.
This condition fails for many useful probability distributions, as for example
$$
\mu_{\alpha} = \frac{1}{Z_{\alpha}} e^{-|x|^{\alpha}}\,dx, \quad x\in \R
$$
where $0<\alpha <2$. There were many attempts to
reveal which inequalities of the log-Sobolev type are the right ones
for these measures when $\alpha \in (1,2)$ (the condition $\alpha\ge 1$ guarantees a spectral gap property
  $\mathrm{Var}_\mu(f)\le C\int |\nabla f|^2d\mu$).
As for now, we do not know of a Sobolev type inequality with  all the good features
of the log-Sobolev inequality (in terms of consequences for  concentration, tensorisation or
semigroup properties). Instead of this, several functional inequalities were proposed,  each of them 
having one of these good features:

--  The Beckner-Lata{\l}a-Oleszkiewicz inequalities  have the tensorisation property
 (when valid for $\mu$ they hold with the same constant for $\mu^{\otimes d}$ for all $d$) and 
yield concentration estimates with decay  $e^{-Kt^{\alpha}}$, for  the $\ell_2$-distance on the products.
They where first mentioned by Beckner  \cite{Beck} for the Gaussian measure, i.e. $\alpha=2$.
Their modified version for the measures $\mu_\alpha$, $\alpha\in (1,2)$ is due to Lata{\l}a-Oleszkiewicz
 \cite{LO}.

-- $F$-Sobolev inequalities of the form 
$$
\int f^2 F \Bigl( \frac{f^2}{\int f^2 \,d \mu }\Bigr) \,d\mu
\le C \int |\nabla f|^2 \,d \mu
$$
where $F$ is some increasing  function, were established for the measures $\mu_\alpha$ and their $d$-dimensional 
analogs in
\cite{Rosen}, \cite{Adams}, with $F(t)=\log(t)^{2-2/\alpha}$ for large $t$. 
The recent developments can be found
in papers \cite{bartcr06iibe}, \cite{RoZeg}, \cite{Wang06}, \cite{Kol}. 
Inequalities of this type imply hyperboundedness of the related semigroups for certain Orlicz norms and under mild conditions
isoperimetric inequalities, see \cite{bartcr06iibe}. In fact, they are closely related with
 the Sobolev inequalities for Orlicz norm
 $
 \| f-\int f \,d \mu\|^2_{\Phi} \le C \int |\nabla f|^2\,d\mu,
 $
 where $\Phi$ is some Orlicz  function. Details about the connections and
 additional semigroups properties appear in \cite{RoZeg}, \cite{Wang06}.
Let us also mention another important work of Wang \cite{Wang00} devoted to  the so-called
super-Poincar{\'e} inequalities. It establishes a correspondence between  $F$-Sobolev and super-Poincar{\'e} inequalities,
and gives consequences in terms of  isoperimetric and Nash inequalities as well as spectral properties of semigroups.

-- Modified  log-Sobolev inequality with cost function $c$
\begin{equation}
\label{ModLSIlog0}
\mbox{\rm Ent}_{\mu} f^2 \le
 \int_{\R^d} f^2 c^{*} \Bigl(
\frac{\nabla f}{f}
\Bigr) \,d\mu,
\end{equation}
where $c^{*}(x)=\sup_y \langle x,y \rangle -c(y)$ is the convex conjugated of some convex cost function $c: \R^d \to \R^{+}$.
The first modified log-Sobolev inequality was introduced by Bobkov and Ledoux  \cite{BoLed97} for the exponential measure, for
$c^*$ quadratic on a small interval around zero and infinite otherwise.
 The main interest of modified log-Sobolev inequalities is to imply improved concentration 
inequalities for product measures as well as corresponding inequalities between entropy and transportation cost,
see  \cite{Tal}, \cite{Led01}, \cite{BGL}, \cite{BaRo}.
Modified log-Sobolev inequalities, with  appropriate cost functions,  have  been known for some time for the measures $\mu_{\alpha}$,
$\alpha \ge 2$, see \cite{BoLed00}. They were established  only recently by Gentil, Guillin and Miclo \cite{gentgm05mlsi}
to  the case $1 \le \alpha <2$ for a functions $c^*_\alpha(x)$
 comparable to $\max(x^2, |x|^{\alpha/(\alpha-1)})$. See \cite{gentgm07mlsi},
\cite{Kol}, \cite{BaRo} for other examples.

\medskip

An isoperimetric inequality is a lower bound of the $\mu$-boundary measure  of sets $\mu^+(\partial A)$
in terms of their measure $\mu(A)$. Recall that for a Borel measure
on a metric space $(X,\rho)$, the boundary measure of a Borel set $A\subset X$ can be defined as the Minkowski
content
  $$\mu^+(\partial A) =\liminf_{h\to 0^+} \frac{\mu\Big(\big\{x\in X\setminus A;\; \rho(x,A) \le h  \big\}\Big)}{h} \cdot$$
The isoperimetric function of a probability measure is defined for $t\in (0,1)$ by  
$$\mathcal I_\mu(t) =\inf \left\{ \mu^+(\partial A);\; \mu(A)=t \right\}.$$
One easily checks that the measure $\mu_\alpha$, $\alpha\ge 1$ satisfy an isoperimetric inequality of the form
$\mathcal I_{\mu_\alpha}(t) \ge \kappa(\alpha) L_\alpha(t)$, where 
 $$ L_\alpha(t)= \min(t,1-t) \log^{1-\frac{1}{\alpha}}\left(\frac{1}{\min(t,1-t)}\right),\quad t\in (0,1).$$
Indeed the sets of minimal boundary measure for given measure are half-lines for log-concave probability measures
on the real line, see e.g. \cite{bobkh97cbis}.
It is well known that isoperimetric inequalities often imply Sobolev type inequalities. Indeed a natural way to try and
unify the above functional inequalities satisfied by $\mu_\alpha$ is to derive them from the above isoperimetric inequality.
Several papers deal with such results (see \cite{Led94}, \cite{BakLed} for the log-Sobolev inequality, 
\cite{Wang00}  for $F$-Sobolev inequalities). The most general result in this direction is given in \cite{Kol} where 
inequalities of the following form (encompassing $F$-Sobolev and modified log-Sobolev) are deduced from isoperimetric 
inequalities: 
\begin{equation}\label{FSI}
\int_{\R^d}f^2 F_{\tau} \Bigl( \frac{f^2}{\int_{\R^d} f^2 \,d \mu }\Bigr) \,d\mu
\le
\int_{\R^d} f^2 c^{*} \Bigl( \Bigl|
\frac{\nabla f}{f}
\Bigr|\Bigr) \,d\mu
+
B \int_{\R^d} f^2\,d\mu.
\end{equation}
However deriving isoperimetric inequalities is hard. In practice one often proves Sobolev inequalities first
and then deduce the isoperimetric inequalities from a method of Ledoux (see \cite{Led94},  \cite{BakLed}, \cite{Wang00},
 \cite{bartcr06iibe}) which applies when the curvature is bounded below to certain Sobolev inequalities with energy 
term $\int |\nabla f|^2 d\mu$.

\medskip
Let us mention a few successful methods to establish Sobolev type inequalities.
On the real line, thanks to Hardy type
inequalities, it is possible to express simple necessary and sufficient conditions for certain Sobolev inequalities
to hold. This technique was first applied to the logarithmic Sobolev inequality by Bobkov and G\"otze \cite{BG}. See e.g.
 \cite{bartr03sipm}, \cite{BZ}, \cite{BaRo} for further applications.

The semigroup method gives Sobolev inequalities by evolution along the semigroup $e^{tL}$ with generator $L=\Delta-\nabla V.\nabla$,
for with $d\mu(x)=e^{-V(x)} dx$ is an invariant measure. It was developed in the abstract framework of diffusion generators 
by Bakry and Emery \cite{bakre85dh}. These authors proved the following celebrated result: if a probability measure $\mu$
on a Riemannian manifold has a density $e^{-V}$ with respect to the Riemannian volume and if for some $K>0$ it holds pointwise 
$ \mathrm{Hess} V+ \mathrm{Ric} \ge K \,\mathrm{Id}$ then
for all smooth $f$ ,
 $$ \ent_\mu(f^2)\le \frac{2}{K} \int |\nabla f|^2 d\mu.$$
Here $\mbox{\rm Ric}$ is the Ricci tensor of $M$.
This result was complemented by the following  theorem of Wang \cite{Wang97,wang01lsic}: denoting by $\rho$ the geodesic distance, if 
$ \mathrm{Hess} V+ \mathrm{Ric} \ge K \,\mathrm{Id}$ with $K\le 0$ and if there exists $\varepsilon>0$ such that
 $$\int e^{\frac{|K|+\varepsilon}{2} \rho(x,x_0)^2} d\mu(x) <+\infty $$
for some $x_0\in M$,  
then $\mu$ satisfies a log-Sobolev inequality.

  In their seminal paper \cite{ottov00gitl} Otto and Villani showed that optimal mass transportation allows to 
derive log-Sobolev inequalities. Their approach was streamlined by Cordero-Erausquin \cite{Cordero}  and extended
in several subsequent papers, see \cite{CGH,CEMCS,AGK}. Let us define the basic objects of optimal transport theory
and refer to the books \cite{Vill,RR} for details. Given $\mu,\nu$ two Borel probability measures on a Polish 
space $X$ and a cost function $c:X\times X\to \R^+$ vanishing on the diagonal, the $c$-transportation cost from $\mu$
to $\nu$ is
 $$ W_c(\mu,\nu)=\inf\left\{\int_{X\times X} c(x,y) \, d\pi(x,y);\; \pi \in \Pi(\mu,\nu)\right\},$$
where $\Pi(\mu,\nu)$ is the set of probability measures on $X\times X$ with first marginal $\mu$ and second marginal $\nu$.
When an optimal $\pi$ exists it is called ``optimal transportation plan''.  When an optimal plan is supported by the graph of a function
$T:X \to X$, then $T$  pushes forward $\mu$ to $\nu$ and is called optimal transport map. The existence and the structure
of such optimal plans and maps is by now quite developed, see \cite{Vill} and the reference therein.

\medskip
The purpose of this paper is twofold. Firstly we present several new sufficient conditions for variants of log-Sobolev
inequalities to hold, in Euclidean space and on Riemannian manifolds.
 The  idea behind many of them is that  measures on $\R^d$  with density $e^{-V}$ 
 verify a variant of the log-Sobolev inequality provided   the lack of convexity of $V$ is balanced by an appropriate
integrability condition. This principle appears clearly in Wang's theorem as well as the following result about log-concave
measures on $\R^d$ (for absolutely continuous measures, this means that the density is of the form $e^{-V}$ where $V$ is convex
with values in $(-\infty ,+\infty ]$):
 every log-concave probability measure $\mu$ on $\mathbb R^d$ such that $\int \exp(\varepsilon|x|^\alpha)\, d\mu<+\infty $
for some $\varepsilon>0$ and $\alpha\ge 1$ satisfies up to constant the same isoperimetric inequality as $\mu_\alpha$,
namely $\mathcal I_\mu\ge \kappa L_\alpha$. 
 This was proved by Bobkov \cite{Bobkov} for
$\alpha\in \{1,2\}$ and was extended in \cite{bart01lcbe} to $\alpha\in[1,2]$ with an argument that actually applies to $\alpha\ge 1$.

A second purpose of this work is to develop the mass transport approach in order to get new and old results
 in a unified manner. To do this we had to introduce several new ways of handling the terms involved in optimal
transport. Let us mention that transportation is intimately linked with the entropy functional, and therefore naturally
yields modified log-Sobolev inequalities. Among other things this paper shows how to recover $F$-Sobolev inequalities,
and therefore isoperimetric inequalities.

Next we describe the structure of the paper and highlight some of the main results and techniques.
Section 2 is devoted to tightening techniques. A functional inequality is tight when it becomes an equality for constant
functions. It is called defective otherwise. Tightness it crucial in applications of Sobolev inequalities to concentration
or to hypercontractivity properties. A classical method of Rothaus allows to transform a defective log-Sobolev inequality
into a tight one, by means of a Poincar\'e inequality.  It does not apply to the modified inequalities.
Theorem \ref{tightening} develops a new simple method for tightening general ``modified $F$-Sobolev inequalities'' \eqref{FSI}.
This result encompasses and simplifies several existing tightness lemma for $F$-Sobolev inequalities.
We also collect known facts about how to derive global Poincar\'e inequalities from local ones.

  Section 3 gives a short account of the consequences of isoperimetric inequalities in terms
of Sobolev type inequalities, with emphasis on the measures satisfying the same isoperimetric 
inequalities as the model measures $\mu_\alpha$. To do this we combine the main result of
\cite{Kol} with our new tightening results.

  In Section 4 we introduce a new variant of the log-Sobolev inequality, which plays a crucial role in the paper.
 For $\tau\in(0,1]$ we say that $\mu$ satisfies $I(\tau)$ if there exists numbers $B,C$ such that 
every smooth function $f$ verifies
\begin{equation*}
\mbox{\rm Ent}_{\mu} f^2 \le
B \int  f^2 \,d \mu+ C \int_{\R^d} \bigl|\nabla f\bigr|^2
 \log^{1-\tau}\Big(e+ \frac{f^2}{\int f^2 \,d\mu}\Big) \,d\mu. \leqno I(\tau)  
\end{equation*}
Further results of the paper show that for $\alpha\in (1,2)$, $\mu_\alpha$ satisfies $I(\tau)$ for $\tau=2-2/\alpha$.
The main result of the section is that $I(\tau)$ implies appropriate $F$-Sobolev inequalities and modified
log-Sobolev inequalities.

In Section 5 we develop the transportation techniques and establish  variants of  log-Sobolev inequalities and  isoperimetric 
inequalities for measures on $\R^d$.
Let  $d\mu=e^{-V(x)}\,dx$ be a probability measure, denote by $f \cdot \mu$ the measure with density $f$ with respect to $\mu$.
As in previous contributions, the starting point is the ``above-tangent lemma'': 
if  a map   $T(x)= x + \theta(x)$ is the optimal transport, for a strictly convex cost, pushing forward
a  probability measure $f \cdot \mu$ to $\mu$ then
\begin{equation}
\label{above-tangent}
\mbox{Ent}_{\mu}  f \le -\int_{\R^d} \bigl<\nabla f , \theta \bigr> \,d\mu
+ \int_{\R^d} \mathcal{D}_V(x,T(x)) f\,d\mu(x),
\end{equation}
where the convexity defect is 
$$
 \mathcal{D}_V(x,y) = -\Big(V(y) -V(x) -\bigl< \nabla V(x), y-x \bigr>\Big).
$$
Under additional integrability assumptions  we prove corresponding  modified log-Sobolev inequalities and Inequalities $I(\tau)$.
The main part of the work consists in estimating the last term involving the  convexity defect. This can be done when
\begin{itemize}
\item[1)] $\mathcal{D}_V(x,y)$ has an upper bound of the type $c_0(x-y)$ and $\mu$ satisfied certain integrability assumption,
\item[2)] $V$  is controlled by a function of the type $G(\nabla V)$
and $\Delta V$ grows  slower than $|\nabla V|^2$,
\item[3)] $V$ satisfies
$
V(x) \le C_1 \bigl< \nabla V(x), x \bigr> + C_2
$
and certain integrability assumption on $\nabla V$,
\item[4)] $V$ is obtained by a perturbation of some convex potential.
\end{itemize}
We recover and extend the Euclidean version of Wang's result. A simple new result asserts that when
$\mathcal  D_V(x,y)$ is upper bounded by $\lambda c(x-y)$ where $c$ is a strictly convex cost and $\lambda\ge0$
then $\mu$ satisfies a defective modified log-Sobolev inequality with cost $c$ provided there exists $\varepsilon>0$
such that
 $$ \int e^{(\lambda+\varepsilon)c(x-y)} d\mu(x)d\mu(y)<+\infty .$$ 
We also extend a theorem of Bobkov on the isoperimetric inequalities
for log-concave measures. Our results imply in particular that when $D^2V \ge -K \, \mathrm{Id}$ and
$\int \exp(\varepsilon|x|^\alpha) d\mu(x)<+\infty $ for some $K\ge 0$, $\varepsilon>0$, $\alpha>2$ then $\mu$ 
satisfies an isoperimetric inequality on the model of $\mu_\alpha$. We generalize this 
result for the Riemannian case in Section 7.

Section 6 provides improved bounds for specific measures on $\R$.
For instance we  recover by transportation techniques the  modified log-Sobolev inequalities satisfied by the exponential measure.
 We propose an interpretation in terms of transport
of  the condition $|f'/f|<c$ that appears in the result of Bobkov and Ledoux. This is related to 
a simple fact in the spirit of the Caffarelli's contraction theorem \cite{caff00mpot}.

 In Section 7 we generalize some results obtained in this paper to Riemannian manifolds.
 We apply the manifold version of (\ref{above-tangent}) obtained by Cordero-Erausquin,
 McCann and Schmuckenschl{\"a}ger in \cite{CEMCS} for quadratic transportation cost.
 We consider a smooth complete connected Riemannian manifold without boundary
$M$  with a  probability measure $\mu = e^{-V}\,d \mbox{\rm vol}$.
In particular, we establish Sobolev type inequalities and isoperimetric inequalities 
  for $\mu$  such that $D^2 V + \mbox{\rm Ric} \ge 0$ and
 $e^{\varepsilon \rho^{\alpha}(x_0,x)} \in L^1(\mu)$, where
$\rho$ is the Riemannian distance, $\varepsilon>0$ and $x_0$
 is an arbitrary point on $M$ and $\alpha\in (1,2]$.

\bigskip
Most  results of the paper apply to   measures with the tail behavior 
of the order $e^{-|x|^{\alpha}} $ with $1 < \alpha \le 2$
(apart from  Section 3,  Subsection 5.4 (Corollary \ref{wangstrong}),  Theorem \ref{th:wangc} and Theorem \ref{BE-bound}). 
Nevertheless, some of our  results  for $\alpha \le 2$ can be adapted to  $\alpha \ge 2$.

Dealing with  $\alpha > 2$ differs from the opposite situation in several respects.
First of all, unlike the case $\alpha < 2$, we don't have $I_{\tau}$-inequality which allows
to prove both $F$-Sobolev and modified log-Sobolev inequality in a suitable form.
Nevertheless, estimating the linear term in the same way as in Lemma \ref{lem:lin1}, Lemma \ref{lem:w},
we can prove in many cases the defective  modified log-Sobolev inequality with $c = |x|^{\alpha}$.
The tightening procedure can be done with the help of Propositions  \ref{prop:rothaus} and
\ref{prop:lp+dls} due to the fact that the modified log-Sobolev inequality for the cost function $|x|^{\alpha}$
is equivalent to the corresponding $q$-log Sobolev inequality with $q=\alpha^*$. However,  in this case one has to prove (or assume) local $q$-Poincar{\'e} 
inequalities. In the case of $\R^d$ and locally bounded potential $V$ this can be shown by Lemma
\ref{lemma:cheeger}, since the Cheeger inequality implies $q$-Poincar{\'e} inequalities for $q > 1$.
 Finally, we note that the reader can easily check that  Theorem \ref{tightMLSI-transport2} a) and Theorem \ref{pcp}
hold also for  $\alpha >2$  in the case of modified log-Sobolev inequality.

\bigskip

\centerline{\bf List of the main objects considered in this paper}

\begin{itemize}
\item Ambient space: We work in the Euclidean space $(\R^d,\langle\cdot,\cdot\rangle,|\cdot|)$
  or on a Riemannian manifold $(M,g)$ for which the geodesic distance is denoted by $\rho$.

\medskip

 \item Duality: If $\alpha >1$ we denote by $\alpha^*$ the number such that $\frac{1}{\alpha}+\frac{1}{\alpha^*}=1$.
  This is consistent with the definition of the convex conjugate (or Fenchel-Legendre transform) of a function $c^*(x)=\sup_y \langle x,y \rangle -c(y)$, since the conjugate of $x \mapsto  |x|^\alpha/\alpha$ is 
  $x \mapsto |x|^{\alpha^*}/\alpha^*$.
 
 \medskip
 
  \item Special cost functions: for $\alpha>1$, $t\in \R$, 
$\displaystyle
c_{ \alpha}(t) =  \left\{
\begin{array}{lcr}
\frac{t^2}{2}\  \quad       \mbox{if} \ |t| \le 1, \\
 \frac{|t|^{\alpha}}{\alpha} +  \frac{\alpha-2}{2\alpha} \ \quad \mbox{if} \ |t| \ge 1.\\
\end{array}
\right.
$

\noindent
Note that $c_\alpha^*=c_{\alpha^*}$. For $\alpha\in (1,2]$, up to multiplicative constants 
$c_\alpha(t) \approx \min(t^2, |t|^\alpha)$ whereas for $\alpha >2$, $c_\alpha(t) \approx \max(t^2, |t|^\alpha)$.

\medskip

  \item Special generalized entropies:
$
 F_{\tau}(t) = \log^\tau(1+t)-\log^\tau(2)
$, $t\ge 0$.

\medskip

   \item Modified log-Sobolev inequality {\bf (MLSI)} for a cost function $c$:
$$
\mbox{\rm Ent}_{\mu} f^2 \le
 \int  f^2 c^{*} \Bigl(
\frac{|\nabla f|}{f}
\Bigr) \,d\mu
$$
More general functions of $\frac{\nabla f}{f}$ are sometimes considered.
   \item $F$-Sobolev inequality {\bf (FSI)}:
   $$
\int  f^2 F \Bigl( \frac{f^2}{\int f^2 \,d \mu }\Bigr) \,d\mu
\le C \int  |\nabla f|^2 \,d \mu
$$
The $q$-$F$-Sobolev inequality {\bf (qFSI)} is defined with the same formula, replacing $f^2$ by $|f|^q$
   and $|\nabla f|^2$ by $|\nabla f|^q$. When $F=\log$ this is the classical log-Sobolev inequality {\bf(LSI)}.

  \item Inequality $I(\tau)$:
$$
\mbox{\rm Ent}_{\mu}(f^2) \le B\int f^2 d\mu+
 C \int  \Bigl| \nabla f \Bigr|^2
\log^{1-\tau}\left(e+ \frac{f^2}{\int  f^2 \,d\mu}\right) \,d\mu.
$$

 \item Poincar{\'e} inequality {\bf (P)} 
$$
\int \Big|f- \int f \,d\mu\Big| ^2 d\mu \le E   \int |\nabla f|^2 d\mu.
$$
For the $q$-Poincar{\'e} inequality {\bf (qP)}, write $q$ instead of $2$.

 \end{itemize}

\smallskip\noindent
{\bf Acknowledgements:} We would like to thanks Dominique Bakry, J{\'e}rome Bertrand, Michel Ledoux, Assaf Naor and Zhongmin Qian 
for useful discussions and for communicating several references to us.
The second author would like to express his gratitude to the research team of the 
Laboratoire de Statistique et Probabilit{\'e}s from the 
 Universit{\'e} Paul Sabatier in Toulouse, where this work was partially done.



\section{How to tighten the inequalities}\label{sec:tight}
The results of this section apply in rather general settings. For simplicity we assume that
$\mu$ is a probability measure on a Riemannian manifold, and is absolutely continuous with 
respect to the volume measure.

\subsection{Translation invariant energies}
The following result of Bobkov and Zegarlinski \cite{BZ} is an extension to $q\neq 2 $
of an argument going back to Rothaus \cite{roth85aiii}.

\begin{proposition}\label{prop:rothaus}
   Let $q\in (1,2]$. Assume that a probability measure $\mu$ satisfies a defective $q$-log-Sobolev inequality
  as well as a $q$-Poincar\'e inequality:
   $$ \ent_\mu(|f|^q) \le C \int |\nabla f|^qd\mu + D \int |f|^q d\mu \quad \mathrm{and} \quad
      \int \Big|f- \int f \,d\mu\Big| ^q d\mu \le E   \int |\nabla f|^qd\mu.$$
 Then it automatically satisfies a tight $q$-logarithmic Sobolev inequality 
   $$ \ent_\mu(|f|^q) \le 16 \big(C+ (D+1)E\big) \int |\nabla f|^qd\mu.$$
\end{proposition}
This is a simple consequence  from the following inequality (see \cite{roth85aiii,BZ} for its proof) 
$$ \ent_\mu(|f|^q) \le 16\Big(\ent_\mu(|f-\int f \,d\mu|^q)+ \int |f-\int f \, d\mu|^q d\mu \Big).$$

\begin{remark}
   For $q>2$ it is not possible to have $ \ent_\mu(|f|^q) \le K \int |\nabla f|^qd\mu,$
as for $f=1+\varepsilon g$ where $\varepsilon\to 0$ the left-hand side behaves like $\varepsilon^2$
 whereas the energy term is of order $\varepsilon^q$.
\end{remark}
\begin{remark}
The  change of functions $f^q=g^2$ turns the $q$-log-Sobolev inequality into a modified-log Sobolev inequality
with function $c^*(t)$ proportional to $|t|^q$.
\end{remark}

\subsection{Modified energies}
The    method of Rothaus  relies on the invariance the energy term $\int |\nabla f|^qd\mu$ 
 under translations $f\mapsto f+ t$, $t\in \mathbb R$.
In general, this property fails for the modified energy  $\int f^2 H(|\nabla f|/f) d\mu$.
This quantity may be very different for  $f$ and $\tilde{f}=f-\mu(f)$. 
This is why another approach is needed.
The next theorem  allows to tighten quite general inequalities. It encompasses several 
tightening results for $F$-Sobolev inequalities given in  \cite{bartcr06iibe}. 
\begin{theorem}
\label{tightening}
Let $H$ be an even function on $\mathbb R$, which is increasing on $\mathbb R^+$ and satisfies $H(0)=0$ and $H(x)\ge c x^2$.
Assume  that there exists $q\ge 2$ such that $x \mapsto H(x)/x^q$ is non-increasing on $(0,+\infty)$.

Let $F:(0,+\infty )\to \mathbb R$ be an non-decreasing function with $F(1)=0$, such that $x\mapsto xF(x)$ is bounded
from below and one of the following properties is verified for some $A>1$

  $(i)$  the function $\Phi(x)=xF(x)$ extends to a $\mathcal C^2$-function on $[0,A^2]$,

  $(ii)$ there exists a constant $d\ge 0$ such that for all $x\in (0, A^2]$, $F(x)\le d(x-1)$,

 \noindent Assume that a probability measure $\mu$ satisfies a defective modified $F$-Sobolev inequality: for all $f$,
$$ \int f^2F\left(\frac{f^2}{\mu(f^2)}\right) \,d \mu \le \int f^2 H\Big(\frac{|\nabla f|}{f} \Big) \, d\mu + D \int f^2 d\mu$$
If $\mu$ also satisfies a Poincar\'e inequality, then there exists a constant $C$ such that 
for every $f$,
$$ \int f^2F\left(\frac{f^2}{\mu(f^2)}\right)  \,d \mu \le C \int f^2 H\Big(\frac{|\nabla f|}{f} \Big) \, d\mu .$$
\end{theorem}

The proof requires some preparation.
\begin{lemma}\label{lem:A}
Let $f$ be a function such that $\int f^2 d\mu=1$. Let  $A>1$. It holds
\begin{equation}\label{i:var}
         \int_{f^2\ge A^2} f^2 d\mu \le \left( \frac{A}{A-1}\right)^2 \var_\mu(f).
\end{equation}
   If $F$ is as in Theorem~\ref{tightening} above, then there exists a constant $\gamma$ depending on $A$ and $F$ only such that 
\begin{equation}\label{i:F}
   \int f^2 F(f^2)\, d\mu \le \gamma \var_\mu(f)+\int_{f^2 \ge A^2}  f^2 F(f^2)\, d\mu.
\end{equation}
\end{lemma}

\begin{proof}
Since $\var_\mu(|f|)\le \var_\mu(f)$ we may assume that $f\ge 0$. In this case, when $f\ge A$
$$ f-\mu(f)\ge f -\mu(f^2)^{\frac12}= f-1 \ge f- \frac{f}{A} = \frac{A-1}{A} f.$$
Inequality  \eqref{i:var} follows by integration.

Next we establish Inequality \eqref{i:F} when $F$ satisfies Hypothesis $(i)$ of Theorem~\ref{tightening}.
 Let $f\ge 0$ with $\mu(f^2)=1$.
 Noting that $\Phi(1)=0$ and $\Phi'(1)\ge 0$, we have by Taylor's formula
   \begin{eqnarray*}
       \int f^2 F(f^2)\, d\mu &=& \int \Big(\Phi(f^2)-\Phi(1)-\Phi'(1)(f^2-1)\Big) \, d\mu  \\
      &\le & \int_{f^2<A^2} \Big(\max_{[0,A^2]} \Phi''\Big)  \frac{(f^2-1)^2}{2} \, d\mu+ \int_{f^2\ge A^2} \Big(\Phi(f^2)-\Phi(1)-\Phi'(1)(f^2-1) \Big)\, d\mu  \\
      &\le &  \Big(\max_{[0,A^2]} \Phi''\Big)  \frac{(A+1)^2}{2} \int (f-1)^2 d\mu +  \int_{f^2\ge A^2} \Phi(f^2) \, d\mu \\
      &\le&   \Big(\max_{[0,A^2]} \Phi''\Big) (A+1)^2\var_\mu(f)+  \int_{f^2\ge A^2} f^2 F(f^2) \, d\mu,
   \end{eqnarray*}
  where the latter   inequality follows from the bound
  $\int \big(f-\mu(f^2)^\frac12\big)^2 \, d\mu \le 2 \var_\mu(f)$, which is readily checked
  by expanding the square.

Finally, if the function $F$ satisfies Hypothesis $(ii)$ of Theorem~\ref{tightening} we simply observe that
 $ \int f^2 F(f^2)\, d\mu = \int f^2 F(f^2)-d(f^2-1) \, d\mu$
and note that on $\{f^2<A^2\}$,
$$ f^2 F(f^2)-d(f^2-1) \le f^2 d(f^2-1) -d (f^2-1)= d (f^2-1)^2 \le d (A+1)^2(f-1)^2.$$
The claim follows by the same method.
\end{proof}

\begin{lemma} \label{lem:B}
   Consider functions $F$ and  $H$  as in Theorem~\ref{tightening} and set $m=-\inf_{t\in(0,1]} tF(t) \ge 0$.
Let $\mu$ be  a probability measure $\mu$ such that for all $f$,
$$ \int f^2 F\left(\frac{f^2}{\mu(f^2)}\right) \, d\mu \le \int f^2 H\Big(\frac{|\nabla f|}{f} \Big) \, d\mu + D \int f^2 d\mu.$$
 Then for $\eta >0$ and  all functions $f$ with $\mu(f^2)=1$ it holds
 $$\int_{f^2\ge (1+2\eta)^q} f^2 F(f^2)\, d\mu \le  \int f^2 H\Big(\frac{2|\nabla f|}{f} \Big) \, d\mu
    + (D+m) \int_{f^2\ge (1+\eta)^q} f^2 d\mu.$$
\end{lemma}
\begin{proof}
It is enough to work with non-negative functions. The change of function $f^2=g^q$ yields
$$ \int g^q F\left(\frac{g^q}{\mu(g^q)}\right) \, d\mu   \le \int g^q H\Big(\frac{q|\nabla g|}{2g} \Big) \, d\mu + D \int g^q d\mu.$$
Since for $t>0$, $tF(t)\ge tF_+(t)-m$, 
  we get 
  \begin{equation}\label{eq:lsG}
   \int g^q F_+\left(\frac{g^q}{\mu(g^q)}\right) \, d\mu \le \int g^q H\Big(\frac{q|\nabla g|}{2g} \Big) \, d\mu 
     + (D+m) \int g^q d\mu.
  \end{equation}   
  Given a non-negative function $\varphi$ with $\mu(\varphi^q)=1$, we apply the latter inequality to 
  $g=\theta(\varphi)$ where for $x\ge 0$
$$\theta(x)=\frac{1+2\eta}{\eta} (x-1-\eta) \1_{x\in[1+\eta,1+2\eta)}+x\1_{x\ge 1+2\eta}.$$
  Obviously for $x\ge 0$,  $\theta(x)\le x\1_{x\ge 1+\eta }\le x$. Hence
  $$ \mu(g^q)\le \mu(\varphi^q \1_{\varphi\ge 1+\eta})\le \mu(\varphi^q)=1.$$
  This estimate, together with the fact that $\varphi=g$ when $\gamma\ge 1+2\eta$,  yields
  $$ \int g^q F_+\left(\frac{g^q}{\mu(g^q)}\right) \, d\mu
             \ge \int_{\varphi\ge 1+2\eta} \varphi^q F_+\left(\frac{\varphi^q}{\mu(g^q)}\right) \, d\mu 
   \ge \int_{\varphi\ge 1+2\eta} \varphi^q F_+\left(\frac{\varphi^q}{\mu(\varphi^q)}\right) \, d\mu 
   = \int_{\varphi\ge 1+2\eta} \varphi^q F (\varphi^q) \, d\mu.$$
   Finally, since $\nabla g=0$ when $\varphi<1+\eta$, and $|\nabla g|\le 2 |\nabla \varphi|$ when $\varphi\ge 1+\eta$
 $$
   \int g^q H\Big(\frac{q|\nabla g|}{2g} \Big) \, d\mu 
                                \le   \int_{\varphi\ge 1+\eta} g^q H\Big(\frac{q|\nabla \varphi|}{g} \Big) \, d\mu  
                                \le  \int_{\varphi\ge 1+\eta} \varphi^q H\Big(\frac{q|\nabla \varphi|}{\varphi} \Big) \, d\mu,
   $$
  where the last inequality follows from $g\le \varphi$ and $x\mapsto x^q H(1/x)$ non-decreasing on $(0,+\infty)$.
  From the above three estimates, Inequality \eqref{eq:lsG} gives for $\varphi$ with $\mu(\varphi^q)=1$
  $$ \int_{\varphi\ge 1+2\eta} \varphi^q F (\varphi^q) \, d\mu 
  \le  \int_{\varphi\ge 1+\eta} \varphi^q H\Big(\frac{q|\nabla \varphi|}{\varphi} \Big) \, d\mu
    + (D+m) \int_{\varphi\ge 1+\eta} \varphi^q\, d\mu.$$
    The claim follows from the change of functions $f^2=\varphi^q$.
\end{proof}
\begin{proof}[Proof of Theorem~\ref{tightening}]
By homogeneity we may assume that $\int f^2 \, d\mu=1$. Let $\eta>0$ such that $A^2=(1+2\eta)^q$.
Combining the previous two lemmas 
\begin{eqnarray*}
  \int f^2 F(f^2)\, d\mu &\le& \gamma\,  \var_\mu(f)+ \int_{f^2\ge (1+2\eta)^q} f^2 F(f^2) \, d\mu \\
   &\le& \gamma \,  \var_\mu(f)+  \int f^2 H\Big(\frac{2|\nabla f|}{f} \Big) \, d\mu
    + (D+m) \int_{f^2\ge (1+\eta)^q} f^2 d\mu \\
    &\le & \left( \gamma+(D+m) \Big(\frac{A}{A-1} \Big)^2\right)  \var_\mu(f)
     + 2^q \int f^2 H\Big(\frac{|\nabla f|}{f} \Big) \, d\mu,
\end{eqnarray*}
where we have used again that $H(x)/x^q$ is non-increasing.
Finally we apply  Poincar\'e's inequality and the bound $H(x)\ge c x^2$,
 $$\var_\mu(f) \le C_P \int |\nabla f|^2 \, d\mu \le \frac{C_P}{c}  \int f^2 H\Big(\frac{|\nabla f|}{f} \Big) \, d\mu.$$
\end{proof}

\subsection{Tightening for free: local inequalities}
Local inequalities are easy to derive for locally bounded potentials by standard perturbation
techniques. In many cases they allow to tighten defective inequalities. They are defined
below.
\begin{definition}
Let $q\ge 1$ and $\mu$ be  a probability measure. One says that $\mu$ satisfies a local
$q$-Poincar\'e inequality if for every $\eta\in (0,1)$, there exists a set $A$ with $\mu(A)\ge \eta$
such that the measure $\mu_A= \frac{\1_A}{\mu(A)}.\mu$ satisfies a $q$-Poincar\'e inequality, meaning that
there exists $C_A<+\infty$ such that for every smooth $f$,
\begin{equation}\label{eq:local}
 \int \Big|f-\mu_A(f) \Big|^q d\mu_A \le C_A \int |\nabla f|^q d\mu_A.
\end{equation}
When $q=2$ we just say that $\mu$ verifies a local Poincar\'e inequality.
\end{definition}

\subsubsection{Isoperimetric inequalities}

The goal of this paragraph is to show how to extend isoperimetric inequalities when they are
known only for sets of small or large measure. The argument is based on local Cheeger's inequalities
(which are equivalent to local $1$-Poincar\'e inequalities). One gets the following convenient result.

\begin{proposition}\label{prop:iso-tight}
Let $I:[0,1/2]\to \mathbb R^+$ be an non-decreasing function with $I(t)>0$ for $t>0$.
Let $\varepsilon \in (0,1/2)$. Assume that a probability measure $\mu=e^{-V(x)}dx$ on $\mathbb R^d$
admits a locally bounded potential $V$ and  
satisfies 
for every set $A$
$$\mu^+(\partial A)\ge I(a),\quad \mathrm{when} \quad a=\min(\mu(A),\mu(A^c))<\varepsilon.$$
Then there exists a constant $c$ such that arbitrary sets satisfy
$\mu^+(\partial A)\ge c\, I(\min(\mu(A),\mu(A^c))).$
\end{proposition}
The proof is based on the following easy fact:
\begin{lemma}\label{lemma:cheeger}
Let $\mu=e^{-V(x)}dx$ be a probability measure on $\mathbb R^d$. Assume that $V$ is locally 
bounded. Then for every $r>0$ there exists a constant $C_r$ such that the measure 
$\mu_{B_r}=\frac{\1_{B_r}}{\mu(B_r)}\cdot \mu$ satisfies for every set $A$,
  $$ \mu_{B_r}^+(\partial A)\ge C_r \min(\mu_{B_r}(A),\mu_{B_r}(A^c)).$$
\end{lemma}
\begin{proof}
First recall that a probability measure $\nu$ satisfies Cheeger's isoperimetric
inequality with constant $c$ means that for every set $c\, \nu^+(\partial A)\ge \min(\nu(A),\nu(A^c))$.
This is equivalent to the functional inequality
$$\int |f-\mathrm{med}_\nu (f)| d\nu \le c\int |\nabla f|\, d\nu.$$
Using the variational expression of the median
  $$ \int |f-\mathrm{med}_\nu (f)| d\nu=\inf_{a\in \mathbb R} \int |f-a|\, d\nu,$$
  one easily checks that the above inequality for $\nu$ can be transfered to any perturbed
   probability $\eta=e^g \cdot \nu$ as
   $$\int |f-\mathrm{med}_\eta (f)| d\eta \le c e^{\sup g-\inf g}\int |\nabla f|\, d\eta.$$
   Since the uniform probability measure on $B_r$ satisfies Cheeger's isoperimetric inequality,
   so does the measure $\mu_{B_r}$ (indeed $V$ is bounded from above and below on $B_r$).
\end{proof}
\begin{proof}[Proof of Proposition \ref{prop:iso-tight}]
Consider an arbitrary set $A$ with $\mu(A)\in[\varepsilon, 1-\varepsilon]$.
It is enough to find a universal constant $C>0$ for which
 $\mu^+(\partial A) \ge C$. To do this, choose $R$ such that $\mu(B_R)= 1-\varepsilon/2$. 
 Plainly $\mu_{B_R}(A)\le (1-\varepsilon)/(1-\varepsilon/2)<1$ and
 $$\mu_{B_R}(A) = \frac{\mu(A)+\mu(B_R)-\mu(A\cap B_R)}{\mu(B_R)}\ge\frac{\varepsilon+1-\frac{\varepsilon}{2}-1}{\mu(B_R)}
      =\frac{\varepsilon}{2-\varepsilon}>0.$$
 By the previous lemma, $\mu_{B_R}$ satisfies Cheeger's isoperimetric inequality.
 Hence there is a constant $K>0$ (depending only on $\varepsilon$ and $\mu$) such that 
 $\mu_{B_R}^+(\partial A)\ge K$. Finally $\mu^+(\partial A)\ge \mu(B_R) \mu_{B_R}^+(\partial A) \ge (1-\varepsilon/2)K.$
\end{proof}
\subsubsection{Sobolev inequalities}
 Next we deal with defective $F$-Sobolev inequalities.
In the case $q=2$ the following result is a consequence of  several existing results in the literature (R\"ockner-Wang
\cite{rockw01wpil} show that a local Poincar\'e inequality implies a weak Poincar\'e inequality, Wang \cite{Wang00} 
shows that a weak Poincar\'e
 inequality and a specific super Poincar\'e inequality implies a Poincar\'e inequality, and that defective
$F$-Sobolev inequalities imply super-Poincar\'e inequalities. See also Aida \cite{aida98upip}.)
However these results  do not provide explicit constants.
The next proposition gives a concrete bound with a straightforward proof.

\begin{proposition}\label{prop:lp+dls}
Let $q>1$. Let $F:(0,\infty ) \to \mathbb R$ be a non-decreasing function with $F(1)=0$, $F(+\infty )=+\infty$
and such that for all $x\in(0,1)$, $xF(x) \ge -M$, where $M\in [0,+\infty )$.
Let $\mu$ be a probability  measure. Assume that   $\mu$ satisfies a local $q$-Poincar\'e inequality \eqref{eq:local},
 and the following defective 
$F$-Sobolev inequality: for all smooth function $f$
$$\int |f|^q F\left( \frac{|f|^q}{\mu\big(|f|^q \big)}\right) \, d\mu \le C \int |\nabla f|^q d\mu + D\int |f|^q d\mu.$$
Then it satisfies the following $q$-Poincar\'e inequality: for all smooth $f$
$$ \int |f-\mu(f)|^q\, d\mu \le 6\left(2^{q-1} K
+ \frac{C}{4(D+M)}\right) \int |\nabla f|^q d\mu,$$
with $K= \kappa\left(\max\Big(\big(1+(2\cdot 3^{1/(q-1)})^{-1} \big)^{-1}, 1-\big(4 F_+^{-1}(4(D+M))\big)^{-1}  \Big)\right)$, where
 $\kappa(r)=\inf\{ c_A;\; \mu(A) \ge r\}$  
 for $r\in (0,1)$, and $F_+^{-1}$ is the generalized left inverse of $F_+=\max(F,0)$.
\end{proposition}

\begin{proof}
Since $F\ge F_+-M$ the hypothesis implies, for all $f$
\begin{equation}\label{eq:Fplus}
\int |f|^q F_+\left( \frac{|f|^q}{\mu\big(|f|^q \big)}\right) \, d\mu \le C \int |\nabla f|^q d\mu + (D+M)\int |f|^q d\mu.
\end{equation} 
 Without loss of generality we consider a function $f$ with $\mu(f)=0$ and $\mu(|f|^q)=1$.
Given a set $A$ to be specified later, we write
\begin{equation}\label{eq:en2}
 1= \int |f-\mu(f)|^q d\mu= \int |f|^q d\mu = \int |f|^q \1_A \, d\mu+ \int |f|^q \1_{A^c} \, d\mu.
\end{equation}
We bound the first term by means of the local $q$-Poincar\'e inequality, noting that $\int f \, d\mu=0$ implies
 that $ \int f \1_A \, d\mu=-\int f \1_{A^c} \, d\mu$.
 By the convexity 
relation $|x+y|^q \le 2^{q-1} (|x|^q+|y|^q)$, we get for any probability measure 
\begin{equation}\label{eq:varid}
 \int |g-\nu(g)|^q d\nu \ge \frac{1}{2^{q-1}} \int |g|^q d\nu - |\nu(g)|^q.
\end{equation}
The local $q$-Poincar\'e inequality hence guarantees 
\begin{eqnarray*}
   \int |f|^q \1_A d\mu &\le & 2^{q-1} \frac{\left|\int f \1_A\, d\mu\right|^q}{\mu(A)^{q-1}}
                            + 2^{q-1} c_A \int |\nabla f|^q \1_A\, d\mu \\
       &=&  2^{q-1} \frac{\left|\int f \1_{A^c}\, d\mu\right|^q}{\mu(A)^{q-1}}
                            + 2^{q-1} c_A \int |\nabla f|^q \1_A\, d\mu \\
     &\le &  2^{q-1} \int |f|^q d\mu\,  \left(\frac{1-\mu(A)}{\mu(A)}\right)^{q-1}+ 2^{q-1}c_A \int |\nabla f|^q d\mu.
 \end{eqnarray*}
The second term in Equation~\eqref{eq:en2} is estimated using  duality, and the defective
 $F_+$-Sobolev inequality \eqref{eq:Fplus}. For a non-negative non-decreasing function $G$ on $\mathbb R^+$ we apply the inequality 
  $xy\le xG(x)+yG^{-1}(y)$ (This is obvious if $y\le G(x)$. If on the contrary $y>G(x)$ then $x\le \inf\{u; \, G(u)\ge y \}=G^{-1}(y)$).
For  $\varepsilon>0$,
 \begin{eqnarray*}
     \int |f|^q \1_{A^c} \, d\mu &=& \varepsilon  \int |f|^q \frac{\1_{A^c}}{\varepsilon} \, d\mu \le
       \varepsilon\int |f|^q F_+\big(|f|^q\big) \, d\mu+\varepsilon \int \frac{\1_{A^c}}{\varepsilon} F_+^{-1}\left(
      \frac{\1_{A^c}}{\varepsilon}\right) \, d\mu \\
     &\le & \varepsilon C \int |\nabla f|^q d\mu + \varepsilon (D+M) \int |f|^q d\mu
 + (1-\mu(A)) F_+^{-1}\Big(\frac{1}{\varepsilon}\Big) .
 \end{eqnarray*}
Using both estimates and recalling that $\int |f|^q d\mu=1$ gives
$$  1 \le  \left(2 \frac{1-\mu(A)}{\mu(A)}\right)^{q-1}+\varepsilon (D+M) + \big(1- \mu(A)\big) F_+^{-1}\Big(\frac{1}{\varepsilon}\Big)
    +  (2^{q-1}C_A+\varepsilon C)  \int |\nabla f|^q d\mu .$$
To conclude we choose $\varepsilon=1/(4(D+M))$, and $A$ large enough to ensure
$ \big(1- \mu(A)\big) F_+^{-1}\Big(\frac{1}{\varepsilon}\Big) \le \frac{1}{4}$ and  
$\left(2 \frac{1-\mu(A)}{\mu(A)}\right)^{q-1} \le \frac13$.
 Using again $\int |f|^q d\mu=1$ we obtain for $f$ with $\int f \, d\mu=0$ 
$$ \int |f|^q d\mu \le 6\Big(2^{q-1 }C_A+ \frac{C}{4(D+M)}\Big) \int |\nabla f|^2 d\mu$$
provided $\mu(A) \ge 1-1/\big(4 F_+^{-1}(4(D+M))\big)$ and $\mu(A) \ge 1/\big(1+(2\cdot 3^{1/(q-1)})^{-1}\big)$.
 Optimizing on such sets yields the claimed result.
\end{proof}

\begin{remark}
When $q=2$ the estimates can be improved since \eqref{eq:varid} can be replaced by the variance identity.
Also when $F=\log$, the duality of entropy may be used to get a more precise bound
 $$ \int |f|^q \1_{A^c}\le \varepsilon \ent_\mu(|f|^q)+\varepsilon \log\left( \int e^{\frac{\1_{A^c}}{\varepsilon}} d\mu\right)
     = \varepsilon\ent_\mu(|f|^q)+ \varepsilon\log \Big(\mu(A)+\big( 1-\mu(A)\big)e^{\frac1\varepsilon}\Big).$$
\end{remark}
\begin{remark}
The translation invariance of the energy term was implicitly but crucially used. If $\mu$ 
satisfies a local Poincar\'e inequality and a defective modified log-Sobolev inequality with
function $H(x)\ge c x^2$ then the above method yields 
$\int f^2 d\mu \le D \int f^2 H(|\nabla f|/f)\, d\mu$ for functions $f$ with $\mu(f)=0$.
\end{remark}

Here is a  direct consequence of Propositions~\ref{prop:lp+dls} and \ref{prop:rothaus}:
\begin{corollary}
\label{11.08}
Let $q\in  (1,2]$. If a probability measure $\mu$ satisfies a defective $q$-log-Sobolev inequality
as well as a local $q$-Poincar\'e inequality, then it satisfies a tight $q$-log-Sobolev 
inequality. 
\end{corollary}

The next classical result yields local Poincar{\'e} inequalities under mild conditions.
\begin{proposition}\label{prop:local}
 Let $(M,g)$ be a connected smooth and complete Riemannian manifold. Let $d\mu(x)=e^{-V(x)} dv(x)$
be a Borel probability measure on $M$ (here $v$ is the Riemannian volume). If $V$ is locally bounded,
then $\mu$ enjoys a local Poincar{\'e} inequality.
\end{proposition}
\begin{proof}
   In Euclidean space, we could proceed like in Lemma~\ref{lemma:cheeger}. In the general, we use the
   following fact, known as the Calabi lemma (see e.g. \cite{bakrq05vctj}): let $x_0\in M$, then $D=M\setminus \mathrm{Cut}(x_0)$
   is an $x_0$- star-shaped domain. Moreover there is a sequence of pre-compact $x_0$-star-shaped domains $D_n$  with smooth boundary 
    such that
   $\bar{D_n}\subset D_{n+1}$ and $D=\bigcup_n D_n$. Since the Neumann Laplacian of a compact manifold with boundary
   has a spectral gap (see e.g. \cite{gallhl90rg}),
   the uniform probability measure on each $D_n$ satisfies a Poincar{\'e} inequality. Next the measure
 $d\mu_{D_n}(x)= \frac{\mathbf 1_{D_n}(x)}
  {\mu(D_n)} \, d\mu(x) = \frac{\mathbf 1_{D_n}(x)}
  {\mu(D_n)} e^{-V(x)} dv(x) $ is a bounded (multiplicative) perturbation of the uniform probability measure on $D_n$.
   It is classical that is therefore inherits the Poincar{\'e} inequality.
  Finally $\lim \mu(D_n)=\mu(\bigcup_n D_n)=1-\mu(\mathrm{Cut}(x_0))=1$ since the cut locus has volume zero. 
\end{proof}


\section{Functional inequalities via isoperimetry }

Isoperimetric inequalities are known to imply Sobolev type inequalities.
Next, we illustrate  this principle for  $F$-Sobolev 
and modified log-Sobolev inequalities.
In this section $d\mu(x) = \rho(x) \,dx$ is a probability measure on $\R^d$ and $\mathcal{I}_{\mu}$ stands for its
 isoperimetric function:
$${\mathcal I}_{\mu}(t) = \inf_{A \subset \R^d: \mu(A)=t} \mu^{+}(\partial A).
$$
The next statement allows to derive general ``modified $F$-Sobolev inequalities'' from 
isoperimetric estimates. The first part of the theorem, dealing with defective inequalities,
was established in  \cite{Kol}. The tight inequality of  the second part 
easily follows from Lemma~\ref{lem:A} and Lemma~\ref{lem:B} 
We deal below with a non-negative convex cost function $c:\R^+ \to \R^+$ with  $c(0)=0$.
We recall that $c$ is called superlinear if $\lim_{x \to \infty} \frac{c(|x|)}{|x|} = \infty$.
\begin{theorem}
\label{mainth}
Assume that $\mu$ has convex support.
Let $c: \R^{+} \to \R^{+}$ be a convex  superlinear non-negative cost  function, such
that for some non-negative
$n :\R^+ \to \R^{+}$ with $\lim_{k \to 0} n(k)=0$ the following holds:
 \begin{equation*}\label{homog}
\mbox{for any $x, k>0$ } \ c(kx) \le n(k) c(x), \quad
 c^{*}(kx)  \le n(k) c^{*}(x).
 \end{equation*}
Let $F$ be  an increasing concave  function on $\R^{+}$ satisfying
$F(1)=0$, $F(+\infty )=+\infty $ and $ \lim_{y \to 0} y F(y) =0$.
Let $\Phi(t) = \sup_{s>0} (st -sF(s) + s)$.
Assume that
 there exist $\delta>0$, $\eta\in(0,1)$  such that 
\begin{equation*}
\label{1dimint}
\int_{0}^{\eta} \Phi \left(  \delta\,  c \Big( \frac{t F(\frac{1}{t})}{{\mathcal I}_{\mu}(t)} \Big) \right)
\,dt < \infty.
\end{equation*}
Then there exist $C,B >0$ such that for every locally Lipschitz $f$
\begin{equation}
\label{defmodLSI}
\int_{\R^d} f^2 F\Bigl(\frac{f^2}{\int_{\R^d} f^2 \,d \mu}\Bigr) \,d \mu
\le
C \int_{\R^d} f^2 c^{*} \Bigl( \Bigl| \frac{\nabla f}{f} \Bigr| \Bigr) \,d \mu
+
B  \int_{\R^d} f^2 \,d\mu.
\end{equation}
If in addition, there exists $q>0$ such that $x \to c^*(x)/x^q$
is non-increasing,  then the inequality can be 
made tight in the following way: the last term $\int f^2d\mu$ can be replaced by $ \mbox{\rm Var}_\mu(f)$.
\end{theorem}

Let us give a concrete example, which is central in our study.
In what follows, $\tilde{F}_\tau$ is any concave increasing function on $\R^+$  vanishing at $0$, behaving like $\log^\tau$ for
large values and with $\lim_{y\to 0} y\tilde F(y)=0$. It can be $F_\tau$ for $\tau\in (0,1]$, but for $\tau >1$ the definition
has to be modified to ensure concavity.
\begin{corollary}\label{coro:isoalpha}
Assume that the probability measure $\mu$ verifies 
\begin{equation}\label{i:isoalpha}
{\mathcal I}_{\mu}(t) \ge  k\, \min(t,1-t)  \log^{1-\frac{1}{\alpha}}\Big(\frac{1}{\min(t,1-t)}\Big), \qquad t\in [0,1].
\end{equation}
  If  $\alpha \in (1, 2]$ then there exists $C$ such that for all $f$
$$
 \int_{\R^d} f^2 F_{2/\alpha^*}\left(\frac{f^2}{\int_{\R^d}f^2 d\mu}\right)\, d\mu \le C \int_{\R^d} |\nabla f|^2 d\mu
\quad \mbox{ and }\quad
\mbox{\rm Ent}_{\mu} \big(f^2\big) \le C \int_{\R^d} f^2 c_{\alpha^*} \Bigl( \frac{\nabla f}{f}\Bigr)  \,d \mu.
$$
 If $\alpha \ge 2$ then there exists $C$ such that all $f$ verify 
$$
 \int_{\R^d} f^2 \tilde F_{2/\alpha^*}\left(\frac{f^2}{\int_{\R^d}f^2 d\mu}\right)\, d\mu \le C \int_{\R^d} |\nabla f|^2 d\mu
\quad \mbox{ and }\quad
\mbox{\rm Ent}_{\mu} \big(|f|^{\alpha^*}\big) \le  C \int_{\R^d} |\nabla f|^{\alpha^*} \,d\mu.
$$
\end{corollary}
\begin{proof}
Applying Theorem~\ref{mainth} to $c(x)=x^2$ and $F=\tilde F_{2/\alpha^*}$ yields defective $F$-Sobolev inequalities.
However \eqref{i:isoalpha} implies  Cheeger's isoperimetric inequality $\mathcal I_\mu(t) \ge k' \min(t,1-t)$.
 Hence $\mu$ satisfies a Poincar\'e's inequality and this allows us to tighten the inequalities by Theorem~\ref{tightening}.

Applying Theorem~\ref{mainth} for $F=\log$ and $c(x)=|x|^\alpha$ shows that for all $f$
 $$\ent_\mu(f^2) \le C \int f^2 \left|\frac{\nabla f}{f}\right|^{\alpha^*} d\mu + B \int f^2 d\mu.$$
When $\alpha\le 2$, there exists a constant $\kappa$ such that $|x|^{\alpha^*} \le \kappa c_{\alpha^*}(x)$. Applying this bound
yields an
 inequality which can be tightened thanks to  Theorem~\ref{tightening} and the Poincar\'e inequality
again. When $\alpha\ge 2$, $\alpha^*\in (1,2]$, making  the change of function $f^2=g^{\alpha^*}$ in the above inequality 
 yields a defective $\alpha^*$-Sobolev inequality.
Cheeger's isoperimetric inequality also implies that  $\mu$ satisfies $\alpha^*$-Poincar{\'e} inequality
 (see, for example, \cite{bobkh97icpp}). 
By Proposition \ref{prop:rothaus} this is enough to tighten the $\alpha^*$-log-Sobolev inequality.
\end{proof}

\begin{remark}
Recall the following fact that we mentioned in the introduction:
every log-concave probability measure on $\mathbb R^d$ such that $\int \exp(\varepsilon|x|^\alpha)\, d\mu<+\infty $
for some $\varepsilon>0$ and $\alpha\ge 1$ satisfies \eqref{i:isoalpha} for some $k>0$.
See also Subsection~\ref{ss:iso} where the log-concavity assumption is weakened.
\end{remark}

\begin{remark}
One can also establish functional inequalities interpolating between the above $F$-Sobolev
inequalities and modified log-Sobolev inequalities. In particular the above theorem implies the
next result, which 
 was proved in \cite{Kol}, with the restriction $1 < \alpha \le 2$: if $\mu$ satisfies \eqref{i:isoalpha}
for some  $\alpha>1$ and if   $\tau\alpha^* \ge 2$ then  there exists $C'$ such that for all $f$
$$
\int_{\R^d} f^2  \tilde F_{\tau} \Bigl( \frac{f^2}{\int_{\R^d} f^2 \,d\mu} \Bigr) \,d\mu
\le
C' \int_{\R^d} f^2 c_{\tau\alpha^*} \left(\Big| \frac{\nabla f}{f}\Big|\right)  \,d \mu.
$$
\end{remark}

\begin{remark}
\label{22.04}
The techniques of \cite{Kol} allow to show that 
every measure $\mu$ satisfying the isoperimetric inequality \eqref{i:isoalpha}
for some $\alpha\in (1,2]$
also verifies the inequality $I(\tau)$ introduced in the next section, when $\tau = 2/\alpha^*$.
\end{remark}


\section{Inequality $I(\tau)$}

In this section we  introduce a new variant of the logarithmic Sobolev inequality.
For $\tau\in(0,1]$ we say that a measure $\mu$ satisfies Inequality $ I(\tau)$ if for some constants $B,C$ and  all $f$
\begin{equation*}
\mbox{\rm Ent}_{\mu} f^2 \le
B \int_{\R^d}  f^2 \,d \mu+ C \int_{\R^d} \bigl|\nabla f\bigr|^2
 \log^{1-\tau}\Big(e+ \frac{f^2}{\int_{\R^d} f^2 \,d\mu}\Big) \,d\mu. \leqno I(\tau)
\end{equation*}

We show next that any  probability measure
satisfying  $I(\tau)$ and a local Poincar{\'e} inequality, automatically satisfies an $F_\tau$-Sobolev
inequality as well as the corresponding modified log-Sobolev inequality. Recall that for  $F_\tau(t)=
 \log^{\tau}(1+t)-\log^\tau(2) $ and that for $\beta\ge 2$, $c_\beta(t)$ is comparable to 
$\max(t^2, t^\beta)$.

\begin{theorem}
\label{FI}
Let $\tau\in(0,1]$ and $\alpha\in (1,2]$ be related by $\tau = \frac{2(\alpha-1)}{\alpha}$.
Let $\mu$ be a probability measure satisfying Inequality $ I(\tau)$.
Then there exist constants $B_i,C_i$ such that for all $f$
\begin{equation}
\label{F-in}
\int_{\R^d} g^2 F_{\tau}\Bigl( \frac{g^2}{\int_{\R^d} g^2 \,d\mu} \Bigr)\,d\mu
\le
B_1 \int_{\R^d}  g^2 \,d \mu+ C_1 \int_{\R^d} \bigl|\nabla g\bigr|^2
 \,d\mu,
\end{equation}
\begin{equation}
\label{defmodLSI3}
\mbox{\rm Ent}_{\mu} f^2 \le
B_2 \int_{\R^d}  f^2 \,d \mu+
C_2 \int_{\R^d} f^2 c_\alpha^* \Bigl( \Bigl|
\frac{\nabla f}{f}
\Bigr|\Bigr) \,d\mu.
\end{equation}
 If  $\mu$ also verifies a local Poincar{\'e} inequality, then
(\ref{F-in}) and (\ref{defmodLSI3}) can be  tightened (i.e. one can take $B_i=0$).
\end{theorem}

\begin{proof} Let  $\tau\in (0,1)$.
First we deduce \eqref{defmodLSI3} from $\mathcal I(\tau)$. Assume as we may that $f$ is non-negative
with $\int f^2 d\mu=1$. Our task is to bound from above the quantity $\int |\nabla f|^2 \log^{1-\tau}(e+f^2)\, d\mu$.
Since $\tau\in (0,1)$, we may apply Young's inequality in the form $xy\le \tau x^{1/\tau}+(1-\tau)y^{1/(1-\tau)}$
and the easy inequality $x\log(e+x)\le x\log x + e$:
\begin{eqnarray*}
    |\nabla f|^2 \log^{1-\tau}(e+f^2) &\le &
     \varepsilon^2 f^2 \left(\tau \left(\frac{|\nabla f|}{\varepsilon f}\right)^{\frac{2}{ \tau}} +
    (1-\tau) \log(e+f^2) \right) \\
    &\le&  \tau \varepsilon^{2(1-\frac{1}{\tau})} f^2 \left| \frac{\nabla f}{f}\right|^{\frac{2}{\tau}} 
       + \varepsilon^2 (1-\tau) f^2 \log f^2 +  \varepsilon^2 (1-\tau) e.
\end{eqnarray*}
Taking integrals and using the fact that up to constants $c_{\alpha^*}(t)$ is comparable to $\max(t^2,|t|^{2/\tau})$, we obtain
that for some constant $B_0$ depending on $\alpha$
$$ \int  |\nabla f|^2 \log^{1-\tau}(e+f^2) \, d\mu \le  \varepsilon^2 (1-\tau) \ent_\mu(f^2) +  \varepsilon^2 (1-\tau) e
     + \tau \varepsilon^{2(1-\frac{1}{\tau})} B_0 \int f^2 c_\alpha^*\left( \frac{|\nabla f|}{f}\right) d\mu.$$
If we choose $\varepsilon>0$ small enough to have $C \varepsilon^2(1-\tau)\le 1/2$, the above inequality can be combined 
with Inequality $\mathcal I (\tau)$ to obtain the defective modified log-Sobolev inequality \eqref{defmodLSI3}.

\smallskip
In order to show that $\mathcal I(\tau)$ implies a defective $F_\tau$-Sobolev inequality, we consider 
$$
\Phi(x) = \frac{x^2}{\log^{1-\tau}(e+x^2)}, \quad x\in \R.
$$
Let us fix a positive  Lipschitz function $g$. We denote by $L$ the Luxembourg norm of $g$ related to $\Phi$:
$$
L =  \inf \left\{\lambda ;\;   \int  \Phi\Bigl(\frac{g} {\lambda}\Bigr)  \,d \mu \le 1  \right\}.
$$
Thus by definition
$ \displaystyle
\int \Phi \Bigl( \frac{g}{L} \Bigr)\,d\mu=1.
$
Since $\Phi(x) \le x^2$, one has
$$
L^2 \le \int  g^2 \,d\mu.$$
Set $f=\varphi(g/L):=\sqrt{\Phi(g/L)}$. Note that $\int f^2  \,d  \mu =1$.
Thus by hypothesis,
\begin{equation}\label{i:smlsi}
\int f^2 \log  f^2 \,d \mu
\le B +  C \int \bigl|\nabla f\bigr|^2 \log^{1-\tau}\left(e+f^2\right)  \,d\mu.
\end{equation}
The left hand side of this inequality equals to
$$
\int f^2 \log  f^2 \,d \mu
  =\int \frac{g^2}{L^2 \log^{1-\tau}(e+g^2/L^2)} \log\left(  \frac{g^2}{L^2 \log^{1-\tau}(e+g^2/L^2)}\right) \, d\mu.
$$
It is not hard to check that there exists a constant $\kappa\ge 0$ depending on $\tau\in(0,1)$ such that for all $x\ge 0$,
$$ \frac{x}{\log^{1-\tau}(e+x)} \log\left( \frac{x}{\log^{1-\tau}(e+x)}\right) \ge  -\kappa +\frac{1}{2} x \log^\tau(e+x),$$
for instance the existence of a finite $\kappa$ for $x\in [0,4]$ is obvious by continuity, whereas for $x\ge 4$
one may use $\sqrt{x}\ge \log(e+x)$ and the bound $x\log x\ge x\log(e+x)-e$.
Hence there are constants $\kappa_1,\kappa_2>0$ depending on $\tau$ such that
\begin{equation}\label{i:left}
   \int f^2 \log  f^2 \,d \mu \ge -\kappa_1 +
   \frac{\kappa_2}{L^2} \int  g^2 F_{\tau}\Bigl(\frac{g^2}{L^2}\Bigr)\,d\mu
\end{equation}
Now let us estimate the gradient term in \eqref{i:smlsi}.
  Recall  that $f = \varphi(g/L)$, where
$$
\varphi(x) = \frac{x}{\log^{\frac{1-\tau}{2}}(e+x^2) }.
$$
Elementary estimates show that there exists $M>0$ such that for every $x\ge 0$
$$
|\varphi'(x)| \le \frac{M}{\log^{\frac{1-\tau}{2}}(e+x^2)}.
$$
Applying this bound together with the estimate $f^2 \le \frac{g^2}{L^2}$ we obtain
$$
\int  |\nabla f|^2 \log^{1-\tau}(e+f^2)\,d\mu \le
\frac{M^2}{L^2} \int  |\nabla g|^2 \frac{ \log^{1-\tau}(e+f^2)}{\log^{1-\tau}\Big(e+\frac{g^2}{L^2}\Big)} \,d\mu
\le \frac{M^2}{L^2} \int  |\nabla g|^2  \, d\mu.
$$
Combining the latter inequality  with \eqref{i:smlsi} and \eqref{i:left} we get that 
$$
\kappa_2 \int g^2 F_{\tau}\Bigl(\frac{g^2}{L^2}\Bigr)\,d\mu
\le
\Bigl(B + \kappa_1 \Bigr) L^2
+
CM^2 \int |\nabla g|^2
\,d\mu.
$$
The claim follows from the estimate $L^2 \le \int g^2 \,d\mu$
and monotonicity of $F_{\tau}$.

If  $\mu$ satisfies a local Poincar{\'e} inequality, then the defective $F_\tau$-Sobolev inequality
is enough to apply Proposition~\ref{prop:lp+dls}. Hence $\mu$ satisfies a Poincar{\'e} inequality. By Theorem~\ref{tightening},
this is enough to tighten both \eqref{F-in} and \eqref{defmodLSI3}. 
\end{proof}


\section{Optimal transportation and functional inequalities.}

The optimal transport theory is  widely represented in surveys and monographs
and the reader can consult \cite{AGS,Vill,RR} for definitions and main results.

\subsection{The above-tangent lemma}

The application of the optimal transportation techniques to 
functional inequalities is based on the following remarkable
estimate called "above-tangent lemma". It has numerous applications
to functional inequalities (especially Sobolev-type inequalities).
From a more general point of view this inequality comes from
convexity  of some special functional (the so-called "displacement
convexity").  This notion has been introduced by McCann in
\cite{McCann97}. For more details about displacement convexity, above-tangent
inequalities and applications, see  \cite{AGS,CGH,Cordero,CENV,CEMCS,AGK,Kol04,BoKol}.

Given a function $V$ on $\R^d$ we define its {\em convexity defect}
\begin{equation}
\mathcal{D}_V(x,y) =-\left( V(y) - V(x) - \bigl<\nabla V(x), y-x \bigr>\right).
\end{equation}

\begin{lemma}
\label{above-tangent2}
Let $g \cdot \mu$ and $h \cdot \mu$ be   probability measures and
$T:\R^d \to \R^d$ be the optimal transportation mapping pushing
forward $g \cdot \mu$ to $h \cdot \mu$ and minimizing the
Kantorovich functional $W_c$ for some strictly convex superlinear
function $c$. Then the following inequality holds:
$$
\mbox{\rm Ent}_{\mu} g \le \mbox{\rm Ent}_{\mu} h + \int_{\R^d} \bigl< x - T(x), \nabla g(x) \bigr> \,d\mu
+ \int_{\R^d} \mathcal{D}_V(x,T(x)) g\,d\mu.
$$
\end{lemma}
\begin{proof}[Sketch of proof]
Without loss of generality one can assume that $g$ and $h$ are smooth and bounded.
By the change of variables formula
$$
\log g = \log h (T) + V - V(T)  + \log \det DT.
$$
Integrating with respect to $g \cdot \mu$ and changing variables, one gets
\begin{align*}
&\mbox{Ent}_{\mu}  g  = \mbox{Ent}_{\mu} h + \int \bigl(V - V(T)  + \log \det DT\bigr) g\,d\mu
\\& = \mbox{Ent}_{\mu} h + \int \bigl<x - T(x), \nabla V(x)\bigr> g\,d \mu
+ \int \mathcal{D}_V(x,T(x)) g\,d \mu  +\int  \log \det DT \ g\,d\mu
\\&
\le \mbox{Ent}_{\mu} h +
\int \bigl< x - T, \nabla g \bigr> \,d\mu
+ \int \mathcal{D}_V(x,T(x)) g\,d \mu
\\&
+ \int \bigl[ \mbox{Tr} (I -DT)  + \log \det DT \bigr] \ g\,d\mu
\end{align*}
The claim follows from the fact that the last integrand is non-positive. 
This is due to the structure of the optimal transport $T$ which ensures that pointwise, $DT$ 
can be diagonalized, with a non-negative spectrum.
\end{proof}
This lemma tells that the convexity type information about the potential $V$, i.e. an estimate of $\mathcal D_V$
in $$ V(y)=V(x)+\langle y-x,\nabla V(x)\rangle -\mathcal{D}_V(x,y),$$
passes to the entropy functional on the space of probability measures
 $$ \ent_\mu(h)\ge \ent_\mu(g)+\int \langle T(x)-x, \nabla g(x)\rangle \, d\mu -\int \mathcal{D}_V(x,T(x)) g\,d\mu.$$
 The second term of the right-hand side is linear in the displacement $\theta(x)=T(x)-x$. It can be thought of as
 the linear part in the tangent approximation of the entropy functional. We will call it the "linear term".

Our aim is to derive modified LSI inequalities using the  "above-tangent" lemma for $g=f^2$ and $h=1$:
\begin{equation}
   \label{eq:at}
   \ent_\mu(f^2) \le 2\int \langle x-T(x), \nabla f(x)\rangle f(x)\, d\mu(x) +\int \mathcal{D}_V(x,T(x)) f(x)^2\,d\mu(x).
\end{equation}
We will show  that the "linear" term in the above inequality
can be estimated by  assuming the integrability of
$\exp(\varepsilon |x|^{p})$ for some $\varepsilon>0,p>1$. Estimating the term involving  $\mathcal{D}_{V}$ is more difficult and can  be done by different methods, under various 
assumptions. When $\mathcal D_V(x,y)\le c(y-x)$, the integral involving $\mathcal D_V$ can be upper-bounded by the transportation
cost from $f^2\cdot \mu$ to $\mu$ when the unit cost is $c$. This argument was already used many times.
 We will see in the next subsections that
estimates of the form  $\mathcal D_V(x,y)\le \varphi(x)+\psi(y)$ are even more convenient.

\subsection{Estimation of the linear term}
The classical estimate is recalled in the next two  lemmas
\begin{lemma}\label{lem:lin1}
Let $c$ be a convex cost function. Assume that $T$ pushes forward $f^2\cdot \mu$ to $\mu$, and
is optimal for the cost $c$. Then for every $\alpha >0$,
$$ \int_{\R^d} 2 \langle \nabla f, x-T(x)\rangle f \, d\mu \le \alpha \int_{\R^d} c^* \left(\frac{-2\nabla f}{\alpha f} \right) f^2
 d\mu + \alpha W_c\big( f^2\cdot\mu,\mu\big).$$
\end{lemma}
\begin{proof}
We simply apply Young's inequality $\langle u,v\rangle \le \alpha c(u)+ \alpha c^*(v/\alpha)$
to $u=T(x)-x$ and $v=-2 \nabla f(x)/f(x)$, and integrate with respect to $f^2.\mu$.
The conclusion comes from $\int c\big(T(x)-x\big) f(x)^2\, d\mu(x)=W_c(f^2\cdot\mu,\mu).$
\end{proof}
It is well known that the transportation cost $W_c$ in the above lemma can be estimated
in terms of the entropy of $f^2$ if $\mu$ has strong integrability properties. This is 
recalled now:
   \begin{lemma}\label{lem:w}
   Let $\mu,\, f\cdot\mu $ and $g\cdot\mu$ be probability measures and $c$ a cost function.
   Then for all $\alpha>0$,
   $$ W_c(f\cdot\mu,g\cdot\mu) \le \alpha \log\left(\int_{\R^d} e^{\frac{c(x,y)}{\alpha}} d\mu(x) d\mu(y)\right)
    +\alpha(\ent_\mu f+\ent_\mu g).$$
   In particular for any Borel sets $A,B$,
   $$ W_c(\mu_A,\mu_B) \le  \alpha \log\left(\int_{\R^d} e^{\frac{c(x,y)}{\alpha}} d\mu(x) d\mu(y)\right)
    +\alpha\log\left(\frac{1}{\mu(A)\mu(B)}\right),$$
where $\mu_{A}$ is the conditional measure
$\mu_A = \frac{\mu|_{A}}{\mu(A)}$.
   \end{lemma}
   \begin{proof}
      We bound the transportation cost from above by using the product coupling:
      \begin{eqnarray*}
         W_c(f\cdot\mu,g\cdot\mu)&=& \inf\left\{ \int c(x,y) \, d\pi(x,y);\;
        \pi  \mbox{ with marginals }  f.\mu \mbox{ and } g.\mu\right\} \\
        &\le & \alpha \int \frac{c(x,y)}{\alpha} f(x)g(y) \, d\mu(x)d\mu(y) 
      \end{eqnarray*}
  The classical inequality $\int \varphi \psi d\nu \le (\int \varphi \, d\nu)\log(\int e^\psi \, d\nu) + \ent_\nu(\varphi)$
  yields 
   $$  W_c(f\cdot\mu,g\cdot\mu) \le \alpha \log\left( \int e^{\frac{c(x,y)}{\alpha}} d\mu(x)d\mu(y)\right)+ 
   \alpha\ent_{\mu\otimes \mu}(f(x)g(y)).$$
   The claim follows.
   \end{proof}

The next result provides a new way to deal with the linear term in the above tangent
inequality, in relation with Inequality $I(\tau)$.
 It is quite flexible, as it does not require the transport to be optimal.
\begin{proposition}
\label{tang-t-est} Let $\alpha\in (1,2], \delta>0$ and let $\mu$ be a probability measure 
such that 
\begin{equation}
\label{exp-tau}
\int_{\R^d} e^{\delta |x|^\alpha}   \,d\mu  < \infty,
\end{equation}
Let  $T$ be a map which  pushes forward a
probability measure  $f^2 \cdot \mu$ to $\mu$.
Then for all $\varepsilon>0$ there exist $C_1, C_2 >0$ depending on $\delta,\alpha$ and the above integral such that
$$
2 \int_{\R^d} \bigl< \nabla f(x),  x - T(x) \bigr>  f(x) \,d\mu(x)
\le
C_1 + C_2 \int_{\R^d} |\nabla f|^2 \log^{1-\tau}(e+f^2) \,d\mu
+ \varepsilon \mbox{\rm Ent}_{\mu} f^2,
$$
where we have set  $\tau=\frac{2\alpha-2}{\alpha}=\frac{2}{\alpha^*}\in (0,1]$.
\end{proposition}
\begin{proof}
Note that
$
2 \int \bigl< \nabla f(x),  x - T(x) \bigr>  f(x) \,d\mu(x)$ is not bigger than
\begin{align*}
\varepsilon \int \frac{f^2(x)}{\log^{1-\tau}\big(e+f^2(x)\big)}
|x-T(x)|^2 \,d\mu(x)
+
\frac{1}{\varepsilon}
\int |\nabla f|^2 \log^{1-\tau}(e+f^2) \,d\mu
\end{align*}
 for arbitrary $\varepsilon$.
Since $|x-T(x)|^2\le 2|x|^2+2|T(x)|^2$ and $1-\tau\ge 0$,  we get
\begin{align}\label{eq:?}
\int \frac{f^2(x)}{\log^{1-\tau}\big(e+f^2(x)\big)}
& |x-T(x)|^2 \,d\mu(x)
\\&
\le
2\int \frac{f^2(x)}{\log^{1-\tau}\big(e+f^2(x)\big)}
|x|^2 \,d\mu(x)
+ 2 \int f^2(x) |T(x)|^2 \,d\mu(x). \nonumber
\end{align}
We apply the inequality $ab\le a\varphi(a)+b\varphi^{-1}(b)$, $a,b\ge 0$
 for the function $\varphi(t)=e^{\frac{\delta}{2} t^{\alpha/2}}-1$, and get
 \begin{eqnarray*}
   && \int \frac{f^2(x)}{\log^{1-\tau}\big(e+f^2(x)\big)} |x|^2 \,d\mu(x) \\ 
   &\le& \int  \frac{f^2}{\log^{1-\tau}(e+f^2)} 
   \left(\frac{2}{\delta} \log\left( 1+   \frac{f^2}{\log^{1-\tau}(e+f^2)}\right) \right)^{\frac{2}{\alpha}}d\mu
   + \int |x|^2 \left(e^{\frac{\delta}{2}|x|^\alpha}-1\right) d\mu(x) \\
    &\le& \int  \frac{f^2}{\log^{1-\tau}(e+f^2)} 
   \left(\frac{2}{\delta} \log\left( 1+   f^2\right) \right)^{\frac{2}{\alpha}}d\mu
   + \int |x|^2 e^{\frac{\delta}{2}|x|^\alpha} d\mu(x).
 \end{eqnarray*}
Using $1+f^2\le e+f^2$ with the relation $\frac{2}{\alpha}+\tau-1=1$ and 
Assumption~\eqref{exp-tau}, we get constants $\kappa_i$ depending on $\alpha,\delta,\mu$ but not on $f$ such that the above quantity
is at most 
$$ \kappa_1 \int f^2  \log(e+f^2) \, d\mu+ \kappa_2 \le \kappa_1 \ent_\mu(f^2)+\kappa_3.$$
The last term in \eqref{eq:?} is controlled by the change of variable formula and the integrability assumption again:
 $$
\int f^2(x) |T(x)|^2 \,d\mu(x) = \int  |x|^2 \,d\mu(x)\le \kappa_4 < \infty.
$$
The proof is complete.
\end{proof}

\subsection{Basic facts about  convexity defect}

In the next two subsections, we will work with potentials $V$ such that there exists a function $c$
such that 
\begin{equation}\label{eq:conv}
 V(x+u)\ge V(x)+\langle u,\nabla V(x) \rangle - c(u).
\end{equation}
In other words, the defect of convexity satisfies $\mathcal D_V(x,y)\le c(y-x)$.
When $c$ is a negative function, $V$ is uniformly convex.
We will focus on the case when $c$ is positive.  In this case we say that $V$ is weakly convex.
This subsection provides concrete examples of such potentials.

\medskip
The first example is as follows: if  $V$  is $\mathcal C^ 2$, the condition 
 $D^2V \ge -\lambda$ for some $\lambda \ge 0$, is equivalent to condition \eqref{eq:conv} with $c(u)=\frac{\lambda}{2}|u|^ 2.$
It is also equivalent to the fact that the function $V(x)+ \frac{\lambda}{2}|x|^ 2$ is convex. 
Next we present other possible conditions, extending the latter.

Choose $c(x)=\|x\|^p$ where $\|\cdot\|$ is a strictly convex norm and $p>1$.
   Note that $p> 2$ is not very interesting in our case since for a smooth $V$
   $$ V(x+u)-V(x)-\langle u, \nabla V(x)\rangle \sim \frac12 D^2V(x).u.u$$
   dominates $-\lambda \|u\|^p$ only when $\lambda=0$ and $V$ is convex since $-|u|^p >>- |u|^2$ for small $u$.
   The case $p\in (1,2)$ contains new examples. We need some preparation.
   
  For $p\in [1,2]$, a norm on a vector space $X$ has a modulus of smoothness of power-type $p$ or for short is $p$-smooth
with  constant $S$ if  for all $x,y\in X$ it satisfies
$$ \| x+y\|^p+ \| x-y\|^p \le 2 \|x\|^p+2S^p\|y\|^p.$$
As shown in \cite{ballcl94sucs}, this formulation is  equivalent to the more standard definition given in \cite{lindtCBS2}.
We need the  following classical fact, see Lemma~4.1 in \cite{naorpss06mcsb}
\begin{lemma}\label{lem:sm}
   Let $(X\|\cdot\|)$ be a Banach space with a $p$-smooth norm with constant $S$. Let $X$ be a random vector with values in
   $X$ such that $E\|Z\|^p<+\infty $. Then
   $$ E\|Z\|^p\le \|EZ \|^p +\frac{S^p}{2^{p-1}-1} E \|Z-EZ\|^p.$$
\end{lemma}
Applying the above lemma when the law of $Z$ is $(1-t)\delta_x+t\delta_y$ for $t\in(0,1]$, $x,y\in X$ yields
$$ (1-t) \|x\|^p +t \|y\|^p \le \|(1-t)x+ty\|^p + \frac{S^p}{2^{p-1}-1} t(1-t) \big((1-t)^p+t^p \big)\|y-x\|^p,$$
which is equivalent to
$$  \|y\|^p- \|x\|^p \le \frac{ \|x+t(y-x)\|^p- \|x\|^p}{t} + \frac{S^p}{2^{p-1}-1} \|y-x\|^p \big(1+o(1)\big).$$
Letting $t$ to zero yields for almost every $x$ and for all $u$
\begin{equation}\label{eq:sm}
   \|x+u\|^p\le  \|x\|^p + \langle u, \nabla\|\cdot\|^p (x) \rangle + \frac{S^p}{2^{p-1}-1}\|u\|^p.
\end{equation}
In this form the meaning of $p$-smoothness is very clear: the application $\|\cdot\|^p$ is not too much above its tangent
map and the distance is measured by $\|u\|^p$.
 Therefore the application  $-\|\cdot\|^p$ is not too much below its tangent.  Hence we obtain

\begin{corollary} Let $p\in (1,2]$ and $\| \cdot\|$ be a $p$-smooth norm on $\R^d$, with constant $S$.  Let $V:\R^d\to\R$ such that
$V(x)+\lambda \|x\|^p$ is convex in $x$. Then for all almost every $x$ and every  $u$
$$  V(x+u)\ge V(x)+ \langle u, \nabla V(x) \rangle -\lambda \frac{S^p}{2^{p-1}-1}  \|u\|^p.$$
\end{corollary}

 For $p\in (1,2]$ 
    the $L_p$-norm on $\mathbb R^d$ denoted as $\|\cdot\|_p$ is $p$-smooth, and the optimal constants $S$ have
been calculated. However, the use of Lemma~\ref{lem:sm} introduces the poor constant $2^{p-1}-1$.
A better result than \eqref{eq:sm} is obtained by hands:
\begin{lemma}
   \label{lem:smlp}
  Let $p\in (1,2]$, then for all $x,u$ in $\mathbb R^d$, 
  $$  \|x+u\|_p^p\le  \|x\|_p^p + \langle u, \nabla\|\cdot\|_p^p (x) \rangle + 2^{2-p}\|u\|_p^p.$$
\end{lemma}
\begin{proof}
   First note that it is enough to prove the inequality in dimension one. Indeed all the term are sums 
   of $n$ corresponding terms involving only one coordinate. The inequality is obvious for $x=0$ and both 
   terms are $p$-homogeneous in $(x,u)$. Hence it is enough to deal with the case $x=\pm 1$. Finally the case
   $(x=-1,u)$ can be deduced  from $(1,-u)$ and all we have to do is to show that for all $u\in \mathbb R$, it holds
   \begin{equation}
      \label{eq:in1}
      |1+u|^p\le 1+pu +2^{2-p} |u|^p.
   \end{equation}
   Actually when $u>-1$ it is even true that $|1+u|^p\le 1+pu +  |u|^p$. To see this we start with the case $u\ge 0$ 
and consider the function $\varphi$ defined on $[0,+\infty )$ by
   $\varphi(u)= 1+pu + u^p- (1+u)^p$. Clearly $\varphi(0)=\varphi'(0)=0$. Moreover $\varphi$ is convex
   since $\varphi''(u)=p(p-1)\big(u^{p-2}-(1+u)^{p-2}  \big)\ge 0$, using $p-2\le 0$. Hence $\varphi$ is nonnegative.

  Next we prove the stronger inequality when $u\in [-1,0]$. Setting $t=-u$ we have to show that the function
   $\psi$ defined on $[0,1]$ by $\psi(t)=1-pt+t^p-(1-t)^p$ is nonnegative. This is clear since $\psi(0)=0$ and
   $\psi'(t)=p \big(u^{p-1}+(1-u)^{p-1}-1 \big)\ge 0$ since $u,1-u$ and $p-1$ are in $[0,1]$ (so $u^{p-1}\ge u$, $(1-u)^{p-1}\ge 1-u$).

   Finally we prove \eqref{eq:in1} when $u=-t\in(-\infty ,-1]$ by studying the function $\xi$ defined on $[1,+\infty )$
   by $\xi(t)=1-pt+2^{2-p}t^p-(t-1)^p$. First $\xi(1)\ge 0$ and we shall prove that $\xi$ is nondecreasing.
   To see this we compute
    $$ \xi'(t)=  p \big(2^{2-p}t^{p-1}-1-(t-1)^{p-1}\big),\quad
   \xi''(t)=p(p-1)\big(2^{2-p}t^{p-2}-(t-1)^{p-2}\big).$$
    The latter quantity is nonpositive on $[1,2]$ and nonnegative on $[2,+\infty)$. Therefore $\xi'$ achieves its minimum at $t=2$
    where $\xi'(2)=0$. So $\xi$ is nondecreasing as claimed. The proof is complete.
\end{proof}
    
\begin{corollary}\label{cor:p}
Let $p\in (1,2]$.  Assume that there exists $\lambda\ge 0$ such that the function $x\mapsto V(x)+\lambda \|x\|_p^p$ is convex.
Then for almost every $x$ and every $u$, 
$$V(x+u)\ge V(x)+ \langle u, \nabla V(x)\rangle -\lambda 2^{2-p} \|u\|_p^p.$$
In other words $\mathcal D_V(x,y)\le \lambda 2^{2-p} \|y-x\|_p^p.$
\end{corollary}

\subsection{Isoperimetric inequalities for weakly convex potentials}\label{ss:iso}
It follows from the work of Bobkov \cite{Bobkov}, extended in \cite{bart01lcbe}, that for log-concave probability measures
on $\mathbb R^d$, the isoperimetric profile is somehow governed by the decay of the measure outside large balls.
The goal of this subsection is to show that the log-concavity assumption may be weakened.
  In what follows we consider a function $c:\R^d\to \R^+$ which may be identically zero or 
strictly convex superlinear .
For every Borel $A$ let us by denote $\mu_{A}$ the conditional measure
$\mu_A = \frac{\mu|_{A}}{\mu(A)}$.
Let  $B_r = \{x: \|x-x_0\| \le r\}$.  The next lemma extends an isoperimetric inequality proved 
by Bobkov for log-concave measures.

\begin{lemma}\label{Bobk2} Let $c:\R^d\to \R^+$ be a strictly convex superlinear cost function.
 Let $\mu=e^{-V}dx$ be a probability measure on $\R^d$  with $\mathcal D_V(x,y)\le c(y-x)$. Then for every $r>0$ 
 and every Borel set $A$,
\begin{eqnarray*}
\lefteqn{\mu(A) \log \frac{1}{\mu(A)}
+
\mu(A^c) \log \frac{1}{\mu(A^c)}
+
  \log {\mu(B_r)}} \\
  &  \le&
  2 r \mu^{+}(\partial A)
+ \mu(A) \cdot W_c(\mu_A,\mu_{B_r})
+ \mu(A^c) \cdot W_c(\mu_{A^c},\mu_{B_r}).
\end{eqnarray*}
\end{lemma}
\begin{proof}
      Let $T$ be the optimal map pushing forward $\big(f /\mu(f)\big).\mu$ to $\mu_{B_r}$ for the cost function $c(y-x)$.
    If we apply Lemma~\ref{above-tangent2}, we get after multiplication by $\int f \, d\mu$ 
  \begin{eqnarray*}
\ent_\mu f &\le&  \log\frac{1}{\mu(B_r)} \int f\, d\mu-\int  \langle T(x)-x, \nabla f \rangle \, d\mu+  \left(\int f\, d\mu\right)  W_c\big(\big(f /\mu(f)\big).\mu ,\mu_{B_r}\big) \\
      &\le&  \log\frac{1}{\mu(B_r)} \int f\, d\mu+r \int |\nabla f |\, d\mu+\int  \langle x-x_0, \nabla f \rangle \, d\mu
        +  \left(\int f\, d\mu\right) W_c\big(\big(f /\mu(f)\big).\mu ,\mu_{B_r}\big) ,
  \end{eqnarray*}
where we have used that $|T(x)-x_0|\le r$ since the range of $T$ is in $B_r$.
If we sum up the latter upper bound on $\ent_\mu f$ with the corresponding one for $\ent_\mu(1-f)$, the terms
$ \int  \langle x-x_0, \nabla f \rangle \, d\mu$ cancel out. The conclusion follows from letting $f$ tend to $\1_A$.
\end{proof}

\begin{proposition}
\label{set-transport}
Let $c(x) = \tilde c(|x|)$ where $\tilde c$ is identically zero or is  a strictly 
convex superlinear cost function,  increasing on $ \R^+$.
 Let $\mu=e^{-V}dx$ be a probability measure on $\R^d$  with $\mathcal D_V(x,y)\le c(y-x)$.
Assume that for some $\varepsilon >0$,  $$\exp \bigl((3+\varepsilon) c(x-y)\bigr) \in L_1(\mu\otimes \mu).$$
Then there exist  $D>0$ and $a_0>0$  such that the following is true:

for any Borel set $A$ such that $a:=\min\big(\mu(A),\mu(A^c)\big)$ verifies 
$a\le a_0$, it holds
$$
 a \log \Bigl( \frac{1}{a} \Bigr)
\le D r \mu^{+}(\partial A),
$$
where $r$ is chosen so that $a= \mu(B_r^c)$.
\end{proposition}
\begin{proof}
Let us assume that $\mu(A)\le 1/2$ (the case $\mu(A^c)<1/2$ follows from the same method since 
 $A$ and its complement play symmetric roles in our estimates). Hence by hypothesis $a=\mu(A)=\mu(B_r^c)$.
We apply Lemma \ref{Bobk2}. Our task is to bound the transportation costs involved in its conclusion.
Set $\eta=1/(3+\varepsilon)$ and $K(\eta)=\eta \log\left( \int \exp(c(y-x)/\eta) \, d\mu(x)d\mu(y) \right)$. It is finite by 
hypothesis. Lemma~\ref{lem:w} gives
\begin{equation}\label{eq:w1}
\mu(A) W_c(\mu_A,\mu_{B_r})  \le K(\eta)\mu(A)+ \eta\mu(A) \log(1/\mu(A)\mu(B_r^c))=
    K(\eta)a+ \eta a \log\frac1a +\eta a\log\frac{1}{1-a}\cdot
 \end{equation}
Applying the corresponding bound for $\mu(A^c) W_c(\mu_{A^c},\mu_{B_r}) $ would give a term of order $1-a$ which is 
too big. To avoid this problem, we consider another coupling. Let $S$ be  the $c$-optimal map  pushing forward 
  $\mu_{A^c\cap B_r^c}$ to $\mu_{A\cap  B_r}$.
We define  the map $T:A^c \to B_r$  by $T(x)=x$ when $x\in A^c\cap B_r$ and by $T(x)=S(x)$ for $x\in A^c\cap B_r^c$.
One readily checks that  $T$ pushes $\mu_{A^c}$ forward to $\mu_{B_r}$ (this uses the relation  $\mu(A^c) = \mu(B_r)$
and its consequence   $\mu(A^c \cap B^c_r) = \mu(A \cap B_r)$). Hence
\begin{eqnarray*}
    W_c(\mu_{A^c},\mu_{B_r}) &\le&  \frac{1}{\mu(A^c)}\int_{A^ c} c\big(S(x)-x\big)  \, d\mu(x) 
        =   \frac{1}{\mu(A^c)}\int_{A^ c\cap B_r^c}  c\big(S(x)-x\big)  \, d\mu(x) \\
   &=&   \frac{\mu(A^c\cap B_r^c)}{\mu(A^c)}    W_c(\mu_{A^c\cap B_r^c},\mu_{A\cap B_r}). 
\end{eqnarray*}
Apply Lemma \ref{lem:w} to the latter transportation cost yields
\begin{eqnarray*}
  \mu(A^c)  W_c(\mu_{A^c},\mu_{B_r}) &\le&   \mu(A^c\cap B_r^c)   \left( K(\eta) + \eta \log \frac{1}{\mu(A^c \cap B_r^c) \mu(A\cap B_r)} \right) \\
   &=&  \mu(A\cap B_r) K(\eta)+2\eta  \mu(A\cap B_r) \log\frac{1}{ \mu(A\cap B_r) }.
\end{eqnarray*}
Note that  $-x \log x$ is increasing for  $x\le 1/e$. Thus  $\mu(A)=a\le 1/e$ ensures that
\begin{equation}\label{eq:w2}
  \mu(A^c)  W_c(\mu_{A^c},\mu_{B_r})  \le K(\eta)a+ 2\eta a \log \frac1a.
\end{equation}
Combining Lemma \ref{Bobk2} with \eqref{eq:w1},  \eqref{eq:w2} and the relation $\eta=1/(3+\varepsilon)$ yields
$$ \frac{\varepsilon}{3+\varepsilon } a\log\frac1a \le 2r \mu^+(\partial A)+ 2K(\eta) a+\eta a\log\frac{1}{1-a} \cdot$$
When $a$ is small enough, $2K(\eta)a +\eta a\log\frac{1}{1-a}$ is less than half of the left hand side, hence $ \frac{\varepsilon}{3+\varepsilon } a\log\frac1a \le 4r \mu^+(\partial A)$.
\end{proof}

\begin{theorem}\label{th:isomain}
Let $c(x)=\tilde c(|x|)$ be identically zero or a strictly convex superlinear cost function on $\R^d$.
Let $d\mu(x)=e^{-V(x)}dx$ be a probability measure on $\R^d$ with for all $x,y$ and some $\varepsilon>0$ 
$$ \mathcal D_V(x,y)\le c(x-y) \quad \mbox{and} \quad \exp\big((3+\varepsilon) c(x-y)\big) \in L_1(\mu\otimes \mu).$$
Let $\psi:\R^+ \to \R^+$ be strictly increasing. Assume in addition that 
  for some $x_0$,
$$\int_{\R^d} e^ {\psi(|x-x_0|)} \, d\mu(x)\le K <\infty,$$
then there exist $D,a_0>$ such that every Borel set
$A$ such that $a:=\min\big(\mu(A),\mu(A^c)\big)\le a_0$ verifies
$$ 
 D \mu^{+}(\partial A) \ge a \frac{\log\frac1a}{\psi^{-1}\big(\log\frac{K}{a}\big)} \cdot
$$
\end{theorem}
\begin{proof}
Proposition  \ref{set-transport} gives $D r \mu^{+}(\partial A) \ge a \log\frac1a$ when $a \le a_0$ is such that
$a=\mu(B_r^c)$. Using Markov's inequality in exponential form gives
$$ a=\mu\big(\{x\in\R^d;\; |x-x_0|>r\}\big) \le K e^{-\psi(r)},$$
hence $r\le \psi^{-1}(\log\frac{K}{a}).$
\end{proof}
\begin{remark}
  The restriction on the value of $a$ may be weakened or removed  by making more precise calculations in concrete 
situations, or in general situation by applying  Proposition \ref{prop:iso-tight}.
\end{remark}

\begin{remark}\label{rem:iso}
 Assume that  $c$ is not the zero function. Jensen's inequality yields
$$ \int e^{(3+\varepsilon) c\big(x-\int y\, d\mu(y)\big)} d\mu(x) 
 \le  \int\int e^{(3+\varepsilon)c(x-y)} d\mu(y)d\mu(x).$$
 Hence the above theorem for  $\psi=(3+\varepsilon)\tilde c$ gives the following result: if $\mu$ satisfies
 $$ \mathcal D_V(x,y)\le c(x-y) \quad \mbox{and} \quad \exp\big((3+\varepsilon) c(x-y)\big) \in L_1(\mu\otimes \mu),$$
then there exist $D,a_0>$ such that every Borel set
$A$ such that for $a\in (0,a_0)$, $\mathcal I_\mu(a)  \ge  D\, a \frac{\log\frac1a}{\tilde c^{-1}\big(
\frac{1}{3+\varepsilon}\log\frac{K}{a}\big)} \cdot
$
\end{remark}

Next we give an application to potentials with Hessian bounded from below and with a strong integrability property.
A similar statement holds when $\int e^{\psi(|x|)} d\mu(x)<\infty $ for an increasing $\psi$ with
 $\lim_{+\infty }\frac{\psi(t)}{t^2}=+\infty $.
\begin{corollary}\label{wangstrong}
   Let $d\mu(x)=e^{-V(x)} dx$ be a probability measure on $\R^d$. Assume that there exits $K\ge 0$, $\varepsilon>0$, 
    $\alpha>2$ and $x_0\in \mathbb R^d$ such that
  $$ D^2V(x)\ge -K \, \mbox{\rm Id},\; x\in \R^d \quad \mathrm{ and }\quad \int_{\R^d} e^{\varepsilon|x-x_0|^\alpha}d\mu(x)<+\infty.$$
There there exists $\kappa>0$ such that  the isoperimetric profile of $\mu$ satisfies
  $$\mathcal I_\mu(t) \ge \kappa \min(t,1-t) \log^{1-\frac{1}{\alpha}}\Big(\frac{1}{\min(t,1-t)}\Big), \quad t\in(0,1).$$
\end{corollary}
\begin{proof}
   By hypothesis $\mathcal D_V(x,y)\le \frac{K}{2}|x-y|^2$. We need to check that $\int \exp\big(\beta|x-y|^2\big) d\mu(x)d\mu(y)$ 
 is finite for some $\beta>3K/2$. However this is true for every $\beta$.  Indeed
 for every $\delta>0$ there is a constant such that for all $x$, $|x|^2\le \delta|x|^\alpha+ N(\alpha,\delta)$ (e.g. using
  Young's inequality $xy\le x^p/p+y^{p^*}/p^*$ for $p=\alpha/2>1$). Hence
 $$ \int \int e^{\beta |x-y|^2} d\mu(x)d\mu(y) \le \left( \int e^{2\beta|x-x_0|^2} d\mu(x)\right)^2
   \le \left( \int e^{2\beta\big(\delta|x-x_0|^\alpha+N(\alpha,\delta)\big)} d\mu(x)\right)^2 $$
is finite by choosing $\delta<\varepsilon/(2\beta)$. Therefore we may apply the previous corollary with $\psi(t)=t^\alpha$.
This gives the claimed isoperimetric inequalities for small values of $t$. Since $V$ is locally bounded we apply 
 Proposition~\ref{prop:iso-tight} to extend the result to all $t\in (0,1)$.
\end{proof}

\begin{remark}
Modified log-Sobolev inequalities are established in the next subsection under the weaker integrability assumption
$\exp((1+\varepsilon)c(x-y)) \in L^1(\mu\otimes \mu)$ (see Theorem \ref{th:wangc}).  Thus, one may ask  whether the results of this subsection
remain valid when $3+\varepsilon$ is replaced by $1+\varepsilon$. 
This is indeed the case for the statement of Remark~\ref{rem:iso} when  $c(x) = \frac{K}{2} |x|^2$.
 If $D^2 V \ge K$ and $\exp((\varepsilon+K /2)|x-y|^2)
\in L_1(\mu\otimes \mu)$, then Wang's result yields a logarithmic Sobolev inequality. But when the Hessian is bounded from below,
Ledoux  \cite{Led94} showed that  an appropriate (Gaussian) isoperimetric inequality follows.
Apart from the factor $3+\varepsilon$, another feature of the method of this subsection is not completely satisfactory:
it does not seem to extend to the Riemannian setting.
\end{remark}

\subsection{Modified LSI via weak convexity and integrability}

In this section we derive log-Sobolev inequalities 
when the potential $V$ satisfies $\mathcal D_V(x,y)\le  c_0(y-x)$, or equivalently
 $$ V(x+u)\ge V(x)+ \langle u,\nabla V(x)\rangle -c_0(u).$$
 If $c_0$ is negative then  $V$ is strictly uniformly convex and 
 log-Sobolev inequalities have been proved (Bakry-Emery \cite{bakre85dh} for quadratic $c_0$,
  Bobkov and Ledoux  \cite{BoLed00} in general). When $c_0$ is positive, $V$ is not convex anymore
  and an additional integrability assumption is needed to balance the convexity
   defect. 

\begin{theorem}\label{th:wangc}
   Let $d\mu(x)=e^{-V(x)}dx$ be a probability measure on $\mathbb R^d$. 
Assume that  there exists $\lambda\ge 0$ and an even strictly convex function $c:\mathbb R \to \mathbb R^+$ with $c(0)=0$
such that for all $x,u$,
$$ V(x+u) \ge V(x)+u\cdot \nabla V(x)- \lambda c(u).$$
If there exists $\varepsilon>0$ such that 
$$ \int_{\mathbb R^{2d}} e^{(\lambda+\varepsilon) c(y-x)} d\mu(x) d\mu(y)<+\infty;$$
then there exists $K_1,K_2, K_3\ge 0$ such that  every nonnegative smooth function $f$ verifies,
$$ \ent_\mu(f^2) \le K_1 \int f^2 \, c^* \left( \frac{\nabla f}{K_2 f}\right) \, d\mu+ K_3 \int f^2 \,d\mu.$$
\end{theorem} 
\begin{proof}
Let $\eta_1,\eta_2\in (0,1)$.
   Assume that $\int f^2\, d\mu=1$.  Let  $T(x)=x+\theta(x)$ be the optimal transport from
 $f^2 \cdot \mu$ to $\mu$ for the unit cost $c(x-y)$.
 Applying  Lemma~\ref{above-tangent2} to $g=f^2$ and $h=1$ and Young's inequality as in Lemma~\ref{lem:lin1}
 gives
   \begin{eqnarray*}
      \ent_\mu(f^2) &\le &  \eta_1 \int  \bigl< \frac{-2\nabla f}{\eta_1 f}, \theta\bigr>   f^2\, d\mu
   +\lambda W_c(f^2 d\mu,\mu) \\
        &\le & \eta_1 \int \varphi\, c^*\left(\frac{2\nabla f}{\eta_1 f}\right) \, d\mu + \eta_1 \int c(\theta) f^2\, d\mu+
         \lambda W_c(f^2\,d\mu,\mu) \\
     & =&  \eta_1 \int \varphi\, c^*\left(\frac{2\nabla f}{\eta_1 f}\right) \, d\mu 
+  (\eta_1+\lambda)  W_c(f^2 d\mu,\mu)\\ 
     &\le&  \eta_1 \int \varphi\, c^*\left(\frac{2\nabla f}{\eta_1 f}\right) \, d\mu 
     +    (1-\eta_2) \ent_\mu(f^2) 
+   (1-\eta_2) \log\left(\int e^{ \frac{\lambda+\eta_1}{1-\eta_2}c(x-y)} \right) \, d\mu(x) d\mu(y),
           \end{eqnarray*}
where the last inequality comes from Lemma~\ref{lem:w} for $\alpha=(1-\eta_2)/(\lambda+\eta_1)$.
  Rearranging the entropy terms and  tuning $\eta_1,\eta_2$ to ensure that $ \frac{\lambda+\eta_1}{1-\eta_2}\le \lambda+
\varepsilon$  completes the proof.
\end{proof}
Theorem~\ref{th:wangc} yields defective modified log-Sobolev inequalities. Under suitable conditions, the methods of 
Section~\ref{sec:tight} allow to tighten them. This is illustrated by two of the following corollaries.

\begin{corollary}
\label{14.09.07}
Let $p\in (1,2]$ and $q=p/(p-1)\in [2,+\infty)$
 be its dual exponent. Let $d\mu(x)=e^{-V(x)}dx$ be a probability measure on
$\mathbb R^d$. Assume that there exists $\lambda\ge 0$ such that the function $x\mapsto V(x)+\lambda \|x\|_p^p$ is convex and that there exists $\varepsilon >0$ such that
$$ \int_{\mathbb R^{2d}} e^{(\lambda2^{2-p}+\varepsilon  )\|x-y\|_p^p}d\mu(x)d\mu(y)<+\infty.$$
Then there exists constants $K_1, K_2$ such that for every nonnegative smooth function $g$ it holds
$$ \ent_\mu(g^q) \le K_1 \int \| \nabla g\|_q^q d\mu+ K_2 \int g^q d\mu.$$
\end{corollary}

\begin{proof}
 The convexity type hypothesis on $V$ and Corollary~\ref{cor:p} ensure  that $\mathcal D_V(x,y)\le  2^{2-p} \|y-x\|_p^p$.
 The integrability condition allows to apply Theorem~\ref{th:wangc} with $c(u)=\|u\|_p^p$ for which
 $c^*(u)=\|u\|_q^q$.
We conclude the change of function $f=g^{\frac{q}{2}}$.
\end{proof}

\begin{corollary}[Wang \cite{Wang97,wang01lsic}]\label{coro:wang}
    Let $d\mu(x)=e^{-V(x)}dx$ be a probability measure on $\mathbb R^d$. 
Assume that $V$ is $\mathcal C^2$ and there exists $\lambda\ge 0$ such that pointwise $D^2V\ge -\lambda \, \mbox{\rm Id}$.
If there exists $ \varepsilon>0$  and $x_0$ such that $\int \exp( \frac{\lambda+\varepsilon}{2}|x-x_0|^2)\,  d\mu(x)<+\infty $
then for some $K$ and all smooth functions
 $$ \ent_\mu(f^2) \le K \int_{\mathbb R^d} |\nabla f |^2d\mu.$$ 
\end{corollary}
\begin{proof}
   Combining  Theorem~\ref{th:wangc} for $c(u)=u^2/2$ and  Corollary~\ref{11.08} gives the claim inequality under
the slightly stronger assumption $\int \exp\big(\frac{\lambda+\varepsilon}{2}|x-y|^2 \big) d\mu(x)d\mu(y)<+\infty $.
 Following the  proof of Theorem~\ref{th:wangc} in our specific context, we come across  
a term $(\eta_1+\lambda)  \int \frac{|x-T(x)|^2}{2}f(x)^2d\mu(x)$ where $T$ is the optimal map from $f^2\cdot \mu$ to $\mu$
for the quadratic cost.
 In order to get the full result we estimate it a bit differently. In particular, the optimality of $T$ is not used.
Since for all $\eta_2>0$, $|x+y|^2 \le (1+\eta_2) |x|^2+ (1+\eta_2^{-1}) |y|^2$:
\begin{eqnarray*}
\lefteqn{\frac{\lambda+\eta_1}{2} \int |x-T(x)|^2 f(x)^2 d\mu(x)} \\
  &\le & 
   \frac{\lambda+\eta_1}{2} \left(\int  (1+\eta_2) |x-x_0|^2 f(x)^2 d\mu(x)  + (1+\eta_2^{-1}) \int |T(x)-x_0|^2 f^2(x) d\mu(x)\right) \\
   &\le & (1-\eta_3)
     \left( \ent_\mu(f^2) +\log \left( \int e^{\frac{(\lambda+\eta_1)(1+\eta_2)}{2(1-\eta_3)}|x-x_0|^2} d\mu(x)\right)\right)
       + \frac{\lambda+\eta_1}{2}(1+\eta_2^{-1}) \int |y-x_0|^2  d\mu(y).   
\end{eqnarray*}
For small enough $\eta_i>0$ the first term is finite. The second one is finite by the stronger integrability condition.
Hence a defective LSI has been proved. It can be tightened since the potential is locally bounded.
\end{proof}

\begin{remark}
  Wang's original proof yields a better control on the constant.  It is based on semigroup interpolation and seems 
  hard to apply for  integrability conditions of the 
   form $\int \exp c(x-y) \, d\mu(x)d\mu(y)<+\infty $ with non quadratic $c$.
   This is possible with the transportation
   approach, but a limitation remains: the function $c$ in the integrability condition is the same as the one
  which controls the lack of convexity of  $V$. Nevertheless when the potential is convex,  $\lambda=0$ and  
  $c$ disappears from the  convexity hypothesis, hence any integrability assumption can be used. This is  
  similar to what happened with applications of Bobkov's isoperimetric inequality.
\end{remark}

The techniques of the above proof also have the advantage to work in more general conditions:

\begin{theorem}
   Let $d\mu(x)=e^{-V(x)}dx$ be a probability measure on $\R^d$. Assume that there exists $K,L\ge 0$ and  $\varepsilon,p>0$ 
and  such that 
$$ D^2V(x)\ge -\big(K+L|x|^p\big) \, \mathrm {Id} \quad 
  \mbox{and} \quad \int_{\R^d} e^{\frac{L+\varepsilon}{p+2}|x|^{p+2}} d\mu(x)<+\infty, $$
then $\mu$ satisfies a log-Sobolev inequality as well as {\bf (q-LSI)} for $q=\frac{p+2}{p+1}$.
\end{theorem}
\begin{proof}
   Applying Taylor's formula with integral remainder
   \begin{eqnarray*}
      \mathcal D_V(x,y) &\le& -\int_0^1 (1-u) \big< D^2V\big( (1-u)x+u y\big)\cdot(y-x),y-x\big> du \\
      &\le & \int_0^1 (1-u) \Big(K+L| (1-u)x+u y|^p \Big) |y-x|^2 du \\
   &\le& \frac{K}{2} |y-x|^2+L |y-x|^2 \left(\frac{(1+\eta)}{p+2} |x|^p+ \frac{N(p,\eta)}{(p+1)(p+2)} |y|^p \right),
   \end{eqnarray*}
where we have applied for $\eta>0$  the bound $|a+b|^p \le (1+\eta) |a|^p +N(p,\eta) |b|^p$ and have computed the integrals. 
Next we apply the bound $|y-x|^2 \le (1+\eta)|x|^2+ (1+\eta^{-1})|y|^2$ and develop all the products.
The terms of the form $|y|^2,|x|^2, |x|^2|y|^p$ or $|x|^p|y|^2$ are controlled by applying Young's inequality 
in the form $a^2 b^p \le \eta a^{p+2} + M(p,\eta) b^{p+2}$ or the similar upper bound of $a^p b^2$ (but each time
the small $\eta$ factor should appear in front of $|x|$).
Eventually 
$$\mathcal D_V(x,y) \le  \frac{L+\varphi(\eta)}{(p+2)} |x|^{p+2}+M_1(p,\eta) |y|^{p+2}+M_2(p,\eta),$$
where $\varphi(\eta)$ tends to zero as $\eta$ does and all other parameters are fixed.
Hence, integrating against the probability measure $f^2\cdot \mu$ and using the change of variables by $T$, 
$$\int \mathcal D_V(x,T(x))f(x)^2 d\mu(x)  \le  \int  \frac{L+\varphi(\eta)}{(p+2)} |x|^{p+2} f(x)^2 d\mu(x) +
M_1(p,\eta) \int |y|^{p+2} d\mu(y)+M_2(p,\eta).$$
In view of the strong integrability of $\mu$ this can be bounded by $(1-\eta_1) \ent_\mu(f^2)+B(\eta_1)$ for $\eta_1>0$ small enough. 
It remains to bound from above the linear term, using for $\alpha\in \{2, p+2\}$ and any $\eta_2>0$ the inequality
$$ 2 \int \big<\frac{\nabla f(x)}{f(x)},x-T(x) \big> f(x)^2 d\mu(x) \le 
    M_3(\alpha,\eta_2) \int
    \left|\frac{\nabla f}{f}\right|^{\alpha^*}  f^2 d\mu+ \eta_2 \int |x-T(x)|^\alpha f(x)^2 d\mu(x)
     .$$
 Hence the techniques already used allow to bound the latter integral by an arbitrary small fraction of the entropy plus a constant.
For $\alpha=2$ we get a defective {\bf(LSI)}, for $\alpha=p+2$ we get a defective
 modified log-Sobolev inequality with cost $t^{p+2}$, or a defective ${\bf (qLSI)}$ by a change of function.
They may be tightened  by Corollary~\ref{11.08}. Indeed $\mu$ has a locally bounded potential, hence it satisfies 
a local Cheeger inequality by Lemma~\ref{lemma:cheeger}, which implies local $q$-Poincar{\'e} (see e.g. \cite{bobkh97icpp}).
\end{proof}

 Note that the  transportation argument was used 
 in  \cite{gentgm05mlsi} to prove a defective modified log-Sobolev inequality adapted to a given log-concave 
 measure on $\mathbb R$. Our contribution here is rather in the tightening techniques of Section~\ref{sec:tight}
 which yield a soft proof of the main results of \cite{gentgm05mlsi,gentgm07mlsi} with slightly relaxed conditions:
   \begin{corollary}
      Let $d\mu(x)=e^{-\Phi(x)}dx/Z$ where $\Phi$ is an even non-negative convex function on $\mathbb R$ with $\Phi(0)=0$.
 Assume in addition
that for some $\alpha, \,\eta,\,x_0>0$ and $x\ge x_0$,   $\Phi(x)\le \alpha x^2$ and $\frac{\Phi(x)}{x^{1+\eta}}$ increases.
Then for all smooth $f$ 
$$\ent_\mu(f^2)\le \int_{\mathbb R}  H\left(\frac{f'}{f} \right) f^2 d\mu,$$
where $H(x)=c_1 x^2$ for $|x|\le c_2$ and $H(x)=c_1 \Phi^*(c_3x)$ otherwise. Here $c_i$ are constants depending on $\Phi$.
   \end{corollary}
   \begin{proof}
Assume as we may that $ \eta\in (0,1]$, and consider the function
$$ h(x)= |x|^{1+\eta}\mathbf1_{|x|<x_0}+ \frac{\Phi(x)}{\Phi(x_0)} x_0^{1+\eta}\mathbf1_{|x|\ge x_0}.$$
By hypothesis $x_0\Phi'(x_0)\ge (1+\eta)\Phi(x_0)$ which ensures convexity of $h$. On easily verifies
that for $x\ge 0$, $h(x)/x^{1+\eta}$ is non-decreasing. This implies that $h^*(x)/x^\beta$ is non-increasing
on $\mathbb R^+$ where $\beta\ge 2$ is the dual exponent of $1+\eta$. Also note  that for small $x$, $h^*(x)=|x|^\beta \le x^2$
whereas for large $x$, $h^*(x)=c_4\Phi^*(c_5 x)\ge c_6 x^2$.
Consequently the function $\tilde h(x)=\max(x^2,h^*(x))$ is bounded above  by the function $H(x)$ of the Corollary for a suitable 
choice of the constants. Moreover $\tilde h(x)/x^\beta$ is non-increasing on $\mathbb R^+$.

We apply Theorem~\ref{th:wangc} with the convex potential $V=\Phi+\log Z$, the cost function $c=h(\cdot/2)$ and 
$\varepsilon=x_0^{-1-\eta}\Phi(x_0)$.
 By convexity and parity 
$$\int e^{\varepsilon h\left(\frac{y-x}{2}\right)} d\mu(x)d\mu(y) \le \left(\int_{\mathbb R} e^{\frac\varepsilon2 h(x)-\Phi(x)}
 \, \frac{dx}{Z} \right)^2,$$
which is finite  since $\varepsilon h$ coincides with $\Phi$ in the large, where $\Phi$ grows at least linearly.
Therefore any smooth function verifies
$$ \ent_\mu(f^2) \le \kappa_1 \int f^2 d\mu + \kappa_2 \int h^*\left(\frac{f'}{\kappa_3 f}\right) f^2d\mu
\le \kappa_1 \int f^2 d\mu + \kappa_4 \int \tilde h\left(\frac{f'}{\kappa_5 f}\right) f^2d\mu.$$
The measure $\mu$ being log-concave, it satisfies a Poincar{\'e} inequality, see \cite{Bobkov}. Therefore 
the latter inequality may be tightened using Theorem~\ref{tightening}.
 \end{proof}

In the case of non quadratic cost functions, the previous method provides tight inequalities only when
a spectral gap inequality is known by other means. To avoid this problem we may also work with 
 Inequalities $I(\tau)$; as explained before they imply $F$-Sobolev inequalities which may be easily tightened
using only local Poincar{\'e} property.

\begin{theorem}
\label{tightMLSI-transport1}
Let $d\mu(x)=e^{-V(x)}dx$ be a probability measure on $\mathbb R^d$.   Assume  that
there exists a strictly convex  superlinear function $c_0$ on $\mathbb R^d$ such
that $V$ satisfies 
$$
V(x+u) \ge V(x)+u\cdot \nabla V(x) - c_0(u), \quad x,u\in \mathbb R^d
$$
and for some $\varepsilon >0$
\begin{equation}
\label{DV-integral}
\int_{\R^{2d}} e^{ (1+\varepsilon) c_0(x-y)} \,d \mu(x) d \mu(y) < \infty.
\end{equation}
If there exists $\eta>0$ and $\alpha\in (1,2]$ such that $\int_{\R^d} e^{\eta |x|^\alpha} d\mu(x)<+\infty$  
then $\mu$ satisfies  Inequality  $I(\tau)$ for $\tau=\frac{2\alpha-2}{\alpha}=\frac{2}{\alpha^*}$: for all smooth $f$,
$$
\mbox{\rm Ent}_{\mu} f^2 \le
B \int_{\R^d}  f^2 \,d \mu+ C \int_{\R^d} \bigl|\nabla f\bigr|^2
\log^{1-\tau}\Big(e +  \frac{f^2}{\int_{\R^d} f^2 \,d\mu}\Big) \,d\mu.
$$
If, in addition,  we assume the local Poincar{\'e} inequality, then $\mu$ satisfies

1) a  modified log-Sobolev inequality with $c=c_{\alpha}$,  

2) an  $F$-inequality with $F=F_{\tau}$.
\end{theorem}
\begin{proof}
Let  $T$ be the optimal transport minimizing $W_{c_0}$ and sending $f^2 \cdot \mu$ to $\mu$.
We apply the ``above tangent'' lemma \ref{above-tangent2} in this situation and estimate the linear term
by  Proposition \ref{tang-t-est}. It remains to estimate
 $\int_{\R^d}\mathcal{D}_{V}(x,T(x))f^2\,d\mu$. Since $T$ minimizes the $c_0$-Kantorovich functional,
one obtains 
\begin{eqnarray*}
\int_{\R^d}\mathcal{D}_{V}(x,T(x))f^2\,d\mu& \le& \int_{\R^d} c_0(x-T(x)) f^2\,d\mu=W_{c_0}(\mu, f^2\cdot \mu)   \\
  &\le& \frac{1}{1+\varepsilon} \log\left( \int e^{(1+\varepsilon)c_0(x-y)}d\mu(x)d\mu(y)\right) +\frac{1}{1+\varepsilon} \ent_\mu(f^2),
\end{eqnarray*}
where the latter inequality follows from Lemma~\ref{lem:w}.
This proves Inequality $I(\tau)$. By Theorem~\ref{FI}, the measure $\mu$ satisfies a defective $F_\tau$-Sobolev 
inequality as well as a defective modified Sobolev inequality with function $c_{\alpha}$. 
However  the local Poincar{\'e} inequality and the defective $F_\tau$-Sobolev 
inequality yield  a Poincar{\'e} inequality, as follows from  Proposition \ref{prop:lp+dls} for $q=2$, $F=F_\tau$. 
This spectral gap inequality 
allows to tighten the two inequalities by Theorem~\ref{tightening}.
\end{proof}

\begin{corollary}
Let the assumptions of Corollary \ref{14.09.07} hold. Assume in addition
  the local Poincar{\'e} inequality. Then $\mu$ satisfies Inequality $I(\tau)$ and an $F$-inequality  with $F=F_{\tau}$, where $\tau = \frac{2p-2}{p}$.
\end{corollary}

\subsection{Modified LSI for  perturbations of convex potentials}

\begin{theorem}
\label{tightMLSI-transport3} Let $\alpha\in (1,2]$. Let $d\mu(x)=e^{-V(x)}dx$ be a probability measure on $\R^d$ such 
 that  $V = V_0 + V_1$, where $V_0$ is convex and $V_1$ is continuously differentiable.
Assume that there exists a function  $p \in L^1(\mu)$ and a constant $\varepsilon>0$ such that
$$
\exp \Bigl( (1+\varepsilon)
\bigl[V_1(x) - \bigl<x, \nabla V_1(x) \bigr> + (V_1 + p)^{*}(\nabla V_1(x)) \bigr] \Bigr) \in
L^1(\mu),
$$
where $
(V_1+p)^{*}$ is the Legendre transform of $V_1+p$.
If for some $\eta>0$, $\int_{\R^d} \exp\big(\eta |x|^\alpha\big)\, d\mu<+\infty $ then $\mu$ satisfies  Inequality $I\big(2/\alpha^*\big)$.
\end{theorem}

\begin{proof} Let $f^2\cdot \mu$ be a probability measure and $T$ be the optimal transport (e.g. for the quadratic cost)
pushing this measure forward to $\mu$. Lemma~\ref{above-tangent2} gives
\begin{equation}\label{i:tangf2}
   \ent_\mu(f^2) \le 2 \int \langle x-T(x),\nabla f(x)\rangle f(x) \, d\mu(x)+ \int \mathcal D_V(x,T(x))f^2(x) \, d\mu(x).
\end{equation}
Since $V_0$ is convex, the convexity defect of $V$ is controlled by the one of $V_1$
$$ \mathcal D_V(x,T(x)) \le  \mathcal D_{V_1}(x,T(x))= 
V_1(x) - V_1(T(x)) - \bigl<\nabla V_1(x), x-T(x)\bigr>.
$$
The  definition of the Legendre transform gives 
$
\langle \nabla V_1(x),T(x)\rangle
\le
\bigl(V_1+p\bigr)(T(x))
+ \bigl(V_1+p\bigr)^{*}(\nabla V_1(x)).
$
Hence 
$$
 \mathcal D_V(x,T(x)) 
\le
V_1(x)- \langle \nabla V_1(x), x\rangle
+ \bigl(V_1+p\bigr)^{*}(\nabla V_1(x)) + p(T(x)).
$$
By the change of variables
$\int p(T(x)) f^2(x) \,d\mu(x) =\int p \,d\mu < \infty$.
By the duality of entropy and the exponential integrability assumption, there exists a constant $C$ such that
$$
\int_{\R^d} \Bigl[
V_1(x)- \langle \nabla V_1(x), x\rangle
+ \bigl(V_1+p\bigr)^{*}(\nabla V_1(x))
\Bigr] f^2 \,d\mu
\le
C + \frac{1}{1+\varepsilon} \mbox{Ent}_{\mu} f^2.
$$
This gives an upper bound on $\int \mathcal D_V(x,T(x)) f^2(x)\, d\mu(x)$ by a constant plus $1/(1+\varepsilon)$ times
the entropy of $f^2$. We apply Proposition~\ref{tang-t-est} in order to bound the remaining term in \eqref{i:tangf2}
by an arbitrarily small multiple of the entropy plus a gradient term. This completes the proof of $I(2/\alpha^*)$.
\end{proof}
Next, we give a better result for  a concrete potential $V_0$.
\begin{theorem}
\label {pcp} Let $\alpha\in(1,2]$. Let $d\mu(x)=e^{-V(x)} dx$ be a probability measure on $\mathbb R^d$ with 
potential  $$V(x) = N\Bigl(\frac{|x|^{\alpha}}{\alpha} + V_1(x)\Bigr), \quad x\in\R^d$$ where $N > 0 $ is a constant.
If  $V_1$ is continuously differentiable and if there exits $C>0 $ and $\delta  < \frac{\alpha}{2+\alpha}$ such that
$$|\nabla V_1(x)| \le  \delta  |x|^{\alpha-1} + C, \quad x\in\R^d$$
then $\mu$ satisfies the  modified log-Sobolev inequality with $c=c_{\alpha}$, 
as well as an  $F_{2/\alpha^*}$-Sobolev inequality.
\end{theorem}
\begin{proof}
Since $V$ is locally bounded, $\mu$ satisfies a local Poincar{\'e} inequality.
In view of Theorem~\ref{FI}, it is enough to establish Inequality $I(\tau)$ for $\tau=2/\alpha^*$.
The scheme of the proof is the same as for the previous theorem: let $T$ pushing forward $f^2 \cdot \mu$ to $\mu$,
then the bound \eqref{i:tangf2} is available.
  First note that there exists a constant $D$ such that 
$$
|V_1(x)| \le \frac{\delta}{\alpha} |x|^{\alpha} +C|x|+D, \quad x\in \R^d.
$$
Hence $\int \exp(\kappa|x|^\alpha) \, d\mu(x)$ is finite provided $\kappa < N(1-\delta)/\alpha$. In particular, by 
Proposition~\ref{tang-t-est} for all $\varepsilon>0$ there are constants $C_i$ such that 
$$
2 \int \bigl< \nabla f(x),  x - T(x) \bigr>  f(x) \,d\mu(x)
\le
C_1 + C_2 \int |\nabla f|^2 \log^{1-\tau}(e+f^2) \,d\mu
+ \frac{\varepsilon}{2} \mbox{\rm Ent}_{\mu} f^2.
$$
 It remains
to show that $\int \mathcal{D}_{V}(x,T(x)) f^2 d\mu \le (1-\varepsilon) \mbox{\rm Ent}_{\mu} f^2+C_3$ for a small enough $\varepsilon>0$.
Set $V_0(x)=|x|^\alpha/\alpha$. By linearity
$
\mathcal{D}_V(x,y) =N \big(\mathcal{D}_{V_0}(x,y) + \mathcal{D}_{V_1}(x,y)\big).
$
Plainly
\begin{eqnarray*}
   \mathcal D_{V_0}(x,y) &=& \frac{|x|^\alpha}{\alpha}- \frac{|y|^\alpha}{\alpha} +\bigl< |x|^{\alpha-2}x,y-x \bigr> \\
               &\le & \frac{1-\alpha}{\alpha}|x|^\alpha- \frac{|y|^\alpha}{\alpha} + |x|^{\alpha-1}|y|  \\
               &\le& \left(\frac{1-\alpha}{\alpha}+ \varepsilon_0\right)|x|^\alpha + N_1(\alpha,\varepsilon_0) |y|^\alpha,
\end{eqnarray*}
for arbitrary $\varepsilon_0>0$, where we have used Young's inequality in the form
 $uv =\eta( u \frac{v}{\eta}) \le \eta \frac{u^{\alpha/(\alpha-1)}}{\alpha/(\alpha-1)}+\eta \frac{(v/\eta)^\alpha}{\alpha}$.
One obtains a  similar estimate for the convexity defect of $V_1$ by using the previous bounds on $|V_1|$, $|\nabla V_1|$,
\begin{eqnarray*}
\mathcal{D}_{V_1}(x,y) &\le& |V_1(x)|+|V_1(y)|+|x|\, |\nabla V_1(x)|+ |y| \, |\nabla V_1(x)| \\
&\le & \Bigl(\frac{\delta(1+\alpha)}{\alpha} + \varepsilon_0 \Bigr)|x|^{\alpha} + N_2(\alpha,\varepsilon_0) |y|^{\alpha} + C_2.
\end{eqnarray*}
Here we have used Young's inequality as before to separate variables in the product term and also to absorb the linear terms 
$|x|\le \eta |x|^\alpha+N_3(\alpha,\eta)$.
Finally
$$ \mathcal{D}_V(x,T(x)) \le \kappa |x|^{\alpha}
      +N_4(\alpha,\varepsilon_0) |T(x)|^{\alpha} + C_3.$$
where $\kappa= N\Bigl( \frac{(1-\alpha)+\delta(1+\alpha)}{\alpha}+2\varepsilon_0\Bigr)$.
Since $\delta<\alpha/(2+\alpha)$ it is possible to find $\varepsilon_0,\varepsilon>0$ small enough so that 
$\int \exp(\frac{\kappa}{1-\varepsilon} |x|^\alpha) d\mu(x)<+\infty $.
Hence the duality of entropy, the change of variable formula and the strong integrability of $\mu$ yield a bound of the 
form $ \int \mathcal{D}_V(x,T(x)) f(x)^2d\mu(x) \le (1-\varepsilon) \ent_\mu(f^2) + C_4$, as needed.
\end{proof}

\subsection{Modified LSI via integration by parts}

The technique developed here is close to the Lyapunov function method (see e.g.  \cite{BCG}).
We estimate the convexity defect by the divergence of a special vector field (usually $x$ or $\nabla V$) and apply integration by parts.

The next lemma follows immediately from integration by parts

\begin{lemma}
\label{IP-lemma}
Let $\omega$ be a locally Lipschitz vector field. Assume that there exist
$t, s, C \in \R$ such that
$$
s V \le (1-t) \bigl< \nabla V, \omega\bigr>- \mbox{\rm div} (\omega) + C.
$$
Then for every smooth $g$
$$
t \int_{\R^d}\bigl<\nabla V, \omega\bigr> g e^{-V}  \,dx
+
s \int_{\R^d} V g e^{-V} \,dx
\le
C \int_{\R^d} g \,d\mu+ \int_{\R^d} \bigl<\nabla g,\omega\bigr> e^{-V}\,dx.
$$
\end{lemma}

\begin{theorem}
\label{tightMLSI-transport2}
 Let $d\mu(x)=e^{-V(x)}dx$ be a probability measure on $\R^d$, such that 
 $\int_{\R^d} e^{\eta |x|^\alpha}d\mu(x)<+\infty $ for some $\eta>0$, $\alpha\in (1,2]$.
\begin{itemize}
\item[a)]
Assume that $V$ is  continuously differentiable,
 that there exist $C_1, C_2, \varepsilon>0 \in \R$  such that
$$
V(x) \le C_1 \bigl<x ,\nabla V(x)\bigr>  + C_2,
$$
$$
\exp \bigl( \varepsilon  |\nabla V| \log^{\frac{1}{\alpha}}|\nabla V| \bigr) \in L^1(\mu)
$$
and $-V \le g$ with $g \in L^1(\mu)$. Then $\mu$ satisfies Inequality $I(2/\alpha^*)$.
\item[b)]
Assume  that $\alpha=2$,  $V$ is twice continuously differentiable,
  $-V \le g$ such that $g \in L^1(\mu)$
and there exist $s_0 >0, 1> t>0$ such that for every $0 < s < s_0$
there exists $C=C(s,t)$ satisfying
$$
s V  \le  (1-t)| \nabla V|^2 -\Delta V + C.
$$
Then $\mu$ satisfies the defective log-Sobolev inequality.

In particular, the result holds if $V$ is bounded from below and
$s V \le (1-t) |\nabla V|^2 -\Delta V + C$  for some $s>0, 1>t>0$
\end{itemize}
\end{theorem}
\begin{proof}
To prove a) we apply a bit more general estimate than the above-tangent lemma. Namely, let $T$ be the optimal transport sending
$f^2 \cdot \mu$ to $\mu$. Then in the same way as above (changing variables, taking 
logarithm and integrating with respect to $\mu$) we get
$$
{\mbox {\rm Ent}}_{\mu} f^2 = \int f^2 V \,d\mu  - \int_{\R^d} V \,d\mu  + \int \log \det DT \ f^2 \,d\mu.
$$
By the concavity of logarithm
$$
{\mbox {\rm Ent}}_{\mu} f^2  \le \int f^2 V \,d\mu  - \int V \,d\mu  + d \log  \Bigl(\int \frac{\mbox{\rm Tr} DT}{d} \ f^2 \,d\mu\Bigr).
$$
First we note that $ -\int V \,d \mu
\le \int g \,d\mu < \infty$.
Applying the assumption of the theorem and integration by parts, we get 
$$
\int f^2 V \,d\mu  \le C_2 + d C_1 + 2C_1 \int f(x) \bigl<x,\nabla f(x)\bigr> \,d\mu(x).
$$
Note that
$$
\int \mbox{\rm Tr} DT \ f^2 \,d\mu
=
- 2 \int \bigl<T,\nabla f\bigr> f \,d\mu +  \int  \bigl<T,\nabla V\bigr> f^2 \,d\mu.
$$
Then we estimate
$$
 \int f(x) \bigl<\nabla f(x), x\bigr> \,d\mu(x) \quad \mbox{and} \quad \int  \bigl<T,\nabla f\bigr> f \,d\mu
$$
exactly in the same way as in Proposition \ref{tang-t-est}.
Next
\begin{equation}
\label{V-T}
\int \bigl< \nabla V, T\bigr> f^2  \,d\mu
\le
\int e^{\delta |T|^{\frac{2}{2-\tau}}} f^2  \,d\mu
+
 N_1 \int |\nabla V| \log^{\frac{2-\tau}{2}}|\nabla V| f^2  \,d\mu + N_2.
\end{equation}
 Using the assumption on $\nabla V$ one can easily estimate the right-hand side by
$N_1 \mbox{\rm Ent}_{\mu} f^2 + C$. It remains to note that logarithm grows slowly than any linear function.
The proof of a) is complete.

For the proof of b) apply Lemma \ref{IP-lemma} with $\omega = \nabla V $.
One obtains
$$
s \int V f^2 \,d\mu
+ t\int |\nabla V|^2 f^2 \,d\mu
\le
C + 2 \int f \bigl<\nabla f,\nabla V\bigr> \,d\mu.
$$
for some $t>0$ and every $0<s<s_0$. By the Cauchy inequality
$$
\frac{t}{2} \int |\nabla V|^2  f^2 \,d\mu
+
s \int V f^2 \,d\mu
\le
C + \frac{4}{t} \int |\nabla f|^2 \,d\mu.
$$
Choosing arbitrary small $s$ we obtain that for every $N>0$
\begin{equation}
\label{14.04}
N \int |\nabla V|^2  f^2 \,d\mu
+
 \int V f^2 \,d\mu
 \le
 C(N) \Bigl(1 +  \int |\nabla f|^2 \,d\mu\Bigr)
\end{equation}
if $C(N)$ is sufficiently big. This gives the desired bound for
$$
\int \Bigl(V(x) - \bigl<\nabla V(x),x-T(x)\bigr> \Bigr) f(x)^2 \,d\mu(x).
$$
Indeed, the latter is not bigger than
$$
\int \Bigl(V(x) +N  |\nabla V(x)|^2 \Bigr) f(x)^2 \,d\mu(x)
+ \frac{4}{N}  \int |T(x)-x|^2 f(x)^2 \,d\mu(x),
$$
where the first term is estimated by (\ref{14.04}) and the second
term one can easily estimate by $C(\varepsilon_0) + \varepsilon_0
\mbox{\rm Ent}_{\mu} f^2$ for arbitrary $\varepsilon_0$ by
choosing appropriate $N$. Finally, $-\int V(T) f^2 \,d \mu =
-\int V\,d\mu \le |g|_{L^1(\mu)}$. The proof of b) is
complete.
\end{proof}

The analog of b) holds also for $1< \alpha <2$. However,  we need
more restrictive assumptions.

\begin{theorem}
\label{20.04} Let $d\mu(x)=e^{-V(x)}dx$ be a probability measure on $\R^d$, such that 
 $\int_{\R^d}  e^{\eta |x|^\alpha}d\mu(x)<+\infty $ for some $\eta>0$, $\alpha\in (1,2]$.
Assume that  $V$ is twice continuously differentiable and bounded from below and that 
there exist $s >0, 1>t>0$ such that for some  $C=C(s,t)$ one has
\begin{equation}
\label{V-growth}
s \max(V,1)^{\tau} \le (1-t) |\nabla V|^2
-\Delta V + C,
\end{equation}
where $\tau=2/\alpha^*\in (0,1]$.
Then $\mu$ satisfies  Inequality $I(\tau)$.
\end{theorem}
\begin{proof}
The proof  is similar to the proof of Theorem \ref{tightMLSI-transport2}, but more involved.
First we multiply (\ref{V-growth}) by $\max(V,1)^{1-\tau}$ and apply integration by parts.
We  get
\begin{align*}
s \int V f^2 & \,d\mu
+
t \int |\nabla V|^2 \max(V,1)^{1-\tau} f^2 \,d\mu
\\&
\le
2 \int  f \bigl<\nabla f, \nabla V \bigr> \max(V,1)^{1-\tau} \,d\mu
+
 (1-\tau) \int \max(V,1)^{-\tau} |\nabla V|^2 f^2\,d\mu
+ C_1
 .
\end{align*}
First we note that
\begin{align*}
&
\int \max(V,1)^{-\tau} |\nabla V|^2 f^2\,d\mu
\le
\int  |\nabla V|^2 f^2\,d\mu \le
-\frac{1}{t} \int (|\nabla V|^2 - \Delta V) f^2 \,d\mu
+ C
\\&
= \frac{2}{t} \int f \bigl<\nabla f,\nabla V\bigr> \,d\mu
+ C.
\end{align*}
Hence one can write
\begin{align*}
s \int V f^2 & \,d\mu
+
t \int |\nabla V|^2 \max(V,1)^{1-\tau} f^2 \,d\mu
\\&
\le
C_2 \int  f |\bigl<\nabla f, \nabla V \bigr>| \max(V,1)^{1-\tau} \,d\mu
+
 C_3.
\end{align*}
Let us estimate the first term in the right-hand side:
\begin{align*}
2 \int  f |\bigl<\nabla f, \nabla V \bigr>| \max(V,1)^{1-\tau} \,d\mu
&
\le
N(\varepsilon)
\int |\nabla f|^2 \log^{1-\tau}(e+f^2)\,d\mu
\\&
+
\varepsilon \int \frac{f^2} {\log^{1-\tau}(e+f^2)}
|\nabla V|^2 \bigl(\max(V,1)\bigr)^{2(1-\tau)}\,d\mu.
\end{align*}
Let $M_{\delta} = \{ x: f^2 \le e^{\delta V}\}$.  Then for every $\delta'>\delta$
there exists $C(\delta,\delta',\tau)$ such that
$$
\int_{M_{\delta}} \frac{f^2} {\log^{1-\tau}(e+f^2)}
|\nabla V|^2 \bigl(\max(V,1)\bigr)^{2(1-\tau)}\,d\mu
\le
C(\delta,\delta',\tau) \int_{M_{\delta}} e^{\delta' V} |\nabla V|^2 \,d\mu.
$$
In the other hand
\begin{align*}
\int_{M^{c}_{\delta}}  \frac{f^2}{ \log^{1-\tau}(e+f^2)} &
|\nabla V|^2 \bigl(\max(V,1)\bigr)^{2(1-\tau)}\,d\mu
\\& \le
C(\delta,\tau, \inf V) \int_{\R^d} |\nabla V|^2 \max(V,1)^{1-\tau} f^2 \,d\mu.
\end{align*}
Choosing a sufficiently small $\delta'$ and $\varepsilon$ one gets the following:
\begin{align*}
s \int_{\R^d} V f^2 & \,d\mu
+
\frac{t}{2} \int |\nabla V|^2 \max(V,1)^{1-\tau} f^2 \,d\mu
\\&
\le
C_4 \int |\nabla f|^2\log^{1-\tau}(e+f^2 )\,d\mu
+
C_5 \int e^{\delta' V} |\nabla V|^2 \,d\mu
+ C_6.
\end{align*}
Let us show that
$$\int e^{\delta' V} |\nabla V|^2 \,d\mu$$
is finite for a sufficiently small $\delta'$. First we note that
$\int e^{\delta' V}  \,d\mu$ is a finite measure for sufficiently small
$\delta'$. This can be easily proved by H{\"o}lder's inequality since $\int \exp(\eta|x|^\alpha)d\mu(x)<\infty $.
 Next, integrating  inequality
$$
t |\nabla V|^2 \le |\nabla V|^2 - \Delta V + C'
$$
over $e^{\delta' V}  \cdot \mu$ and integrating by parts we easily get
the claim.
Thus, we obtain that there exists $\tilde{C}$
depending on $s,\tau, \delta'$ such that
\begin{align}
\label{19.04}
s \int V f^2 & \,d\mu
+
\frac{t}{2} \int |\nabla V|^2 \max(V,1)^{1-\tau} f^2 \,d\mu
\nonumber
\\&
\le
\tilde{C} \int |\nabla f|^2 \log^{1-\tau}(e+f^2 )\,d\mu
+\tilde{C}.
\end{align}
Since the function $V$ is bounded from below, we get the desired bounds for the terms
$\int V f^2  \,d\mu$     and $\int |\nabla V|^2 \max(V,1)^{1-\tau} f^2 \,d\mu$.
The  estimates of
$-\int \bigl< \nabla V(x), x \bigr> f(x)^2 \,d\mu(x)$
and $-\int V(T) f^2 \,d\mu$ are the same as  in
Theorem \ref{tightMLSI-transport2}.
Finally,
\begin{align*}
&
\int\bigl< \nabla V, T\bigr> f^2 \,d\mu
\le
2\int | \nabla V|^2 f^2 \,d\mu
+
2\int | T|^2 f^2 \,d\mu
\\&
\le
2\int | \nabla V|^2 f^2  \max(V,1)^{1-\tau}\,d\mu
+
2\int | x|^2  \,d\mu(x).
\end{align*}
The  latter can be estimated by (\ref{19.04}). The proof is complete.
\end{proof}

\begin{corollary}
\label{11.08.2}
Under assumptions of Theorems \ref{tightMLSI-transport2}, \ref{20.04}
the tight  modified log-Sobolev inequality with $c=c_{\alpha}$, $\alpha = \frac{2}{2-\tau}$
as well as  $F$-inequality with $F=F_{\tau}$ holds. 
\end{corollary}
\begin{proof}
By Theorem \ref{FI} it suffices to prove the local Poincar{\'e} inequality. 
This follows from Proposition~\ref{prop:local}  since the potential $V$ locally bounded.
\end{proof}

\begin{remark}
\label{rem:KS}
Let us compare this result with the known ones. Theorem
\ref{20.04} is not completely  new for $F$-inequalities
This type of  criteria for $F$-inequalities have
been already considered in work of Rosen \cite{Rosen} (note, however, that 
assumptions on the potential from \cite{Rosen} are stronger).
Kusuoka and Stroock \cite{KuSt} proved different types of hyperboundedness
of semigroups using  Lyapunov function techniques.
We note that assumptions on the potential in Theorem \ref{20.04} and in Theorem 
\ref{tightMLSI-transport2} a) can be viewed as special cases of some Lyapunov function-type
assumptions.  Nevertheless,  such kind of criteria are not known for modified
log-Sobolev inequalities. Also the transportation approach for
this kind of results is new.  Some related results can be also found in \cite{Cat},
\cite{bartcr06iibe}, \cite{CL}.

A less general but more beautiful sufficient condition is the following:
$V$ is bounded from below, for some $s>0$
$$
s |V|^{\tau} \le |\nabla V|^2 + C, \ \lim_{|x| \to \infty}\frac{\Delta V(x)}{|\nabla V(x)|^2}=0.
$$
It appears in many works as a sufficient condition for
log-Sobolev type inequalities (see \cite{ISLFr,BZ,RoZeg,BaRo}).
\end{remark}

\begin{corollary}
\label{radialLSI} Let $V$ be a continuously differentiable
function such that $V(tx)$ is convex as a function of $t \in
[0,\infty)$ for every $x$. Assume that $\int_{\R^d}  e^{\eta|x|^\alpha}d\mu(x)<+\infty$ for some $\eta>0$, $\alpha\in(1,2]$ 
and
$$
\exp( \varepsilon  |\nabla V| \log |\nabla V|^{\frac{1}{\alpha}} ) \in L^1(\mu).
$$
Then the tight  modified log-Sobolev inequality with $c=c_{\alpha}$,
as well as the  $F$-inequality with $F=F_{2/\alpha^*}$ hold.
\end{corollary}
\begin{proof} Since $\varphi(t)= V(tx)$ is convex, it holds $\varphi(0)\ge \varphi(1)-(0-1)\varphi'(1)$.
In other words, 
$$
 V(0) \ge V(x) - \bigl<\nabla V(x), x\bigr>.
$$
The result follows from Theorem  \ref{tightMLSI-transport2}, Corollary \ref{11.08.2}.
\end{proof}

%

\section{Improved bounds  in dimension 1}
We start with a precised version of the ``above tangent'' lemma. We omit the proof which is similar to the 
one of Lemma~\ref{above-tangent2}. The only difference is that the term $\theta'-\log(1+\theta')$ is not lower
bounded by 0. The goal of this section is to develop applications of sharper estimates of this quantity.

\begin{lemma}\label{lem:it1}
   Let $d\mu(x)=e^{-V(x)}dx$ be a probability measure on $\mathbb R$, with $V$ smooth.
Let $f\cdot\mu$ and $g \cdot\mu$ be two probability measures with smooth and compactly supported densities $f,g$. 
Let $T(x)=x+\theta(x)$ be the monotone map pushing forward $f\cdot \mu$ to $g\cdot \mu$. Then
\begin{eqnarray*} 
 \ent_\mu(f)+ \int_{\R}  \Big( V\big(x+\theta(x)\big)-V(x)-\theta(x)V'(x)\Big)f(x)\,  d\mu(x) \\
  + \int_{\R} \big(\theta'-\log(1+ \theta') \big)f \,d\mu   = \ent_\mu(g) - 
  \int_{\R} f' \theta \, d\mu.
\end{eqnarray*}
If $d\mu(x)=e^{-V(x)} \mathbf 1_{x\ge 0} dx$ where $V$ is smooth and convex, and $f,g$ are smooth with finite entropy, then
provided $\lim_{+\infty } f\theta e^{-V}=0$, the equality is valid with an additional term $-f(0)\theta(0)e^{-V(0)}$ on the 
right-hand side.
\end{lemma}


\subsection{Inequalities for the exponential law}
Let $d\mu(t)= e^{-t} \mathbf1_{t>0}\, dt$ be the exponential measure. Our goal is to provide a simple transportation
proof of the modified log-Sobolev inequality for $\mu$ due to Bobkov and Ledoux \cite{BoLed97}. We also discuss related
transportation cost inequalities. 

We start with recalling useful Sobolev type inequalities for $\mu$. The first part of the next lemma is 
 a particular case of a result of Bobkov and Houdr\'e \cite{bobkh97icpp}, for which we provide a streamlined proof.
The second part is Lemma~2.2 of Talagrand's paper \cite{Tal}.
\begin{lemma}\label{lem:bh} 
$\,$

1)  Let $N: \mathbb R\to [0,+\infty)$ be an even differentiable convex function with $N(0)=0$.
Assume that there exists $c\ge 1$ such that $xN'(x)\le cN(x)$ for all $x$.  
 Let $\varphi:[0,+\infty) \to [0,+\infty)$ be a locally Lipschitz function with $\varphi(0)=0$, 
then
$$ \int_{\mathbb R^+} N(\varphi) \, d\mu \le \int_{\mathbb R^+} N(c\varphi')\,d\mu.$$

2) Let $M(x)=x-\log(1+x),\, x>-1$ and $S(x)=x-1+e^{-x}$ and $\alpha\in (0,1)$. Let  $\varphi:[0,+\infty) \to [0,+\infty)$ as
above with $\varphi(0)=0$ and $\varphi'\ge -1$, then
  $$ \frac{1-\alpha}{\alpha}\int_{\mathbb R^+} S(\alpha\varphi) \, d\mu \le \int_{\mathbb R^+} M(\varphi')\,d\mu.$$
\end{lemma}

\begin{proof}
 First we assume that $\varphi$ is also bounded. For $a>0$ an integration by parts yields
 $$ \int_{0}^a N\big(\varphi(x)\big) e^{-x}dx -N\big(\varphi(a)\big)e^{-a} = \int_{0}^a  \varphi'(x) N'\big(\varphi(x)\big) e^{-x}dx.$$
 Let $N^*$ be the Legendre transform of $N$, defined by $N^*(v)=\sup_u \{uv-N(u)\}$. Then the following inequalities hold pointwise:
 \begin{eqnarray*}
  \varphi' N'(\varphi)  &\le&  \frac{1}{c} \Big( N(c\varphi') + N^*\big(N'(\varphi)\big) \Big) 
    =  \frac{1}{c} \big( N(c\varphi') + \varphi N'(\varphi)-N(\varphi)\big)\\
    & \le & \frac{1}{c} \big( N(c\varphi') + (c-1) N(\varphi)\big).
  \end{eqnarray*}
Plugging this inequality in the above integral equality and rearranging gives
  $$ \int_{0}^a N\big(\varphi(x)\big) e^{-x}dx -cN\big(\varphi(a)\big)e^{-a} \le \int_{0}^a  N\big(c\varphi'(x)\big) e^{-x}dx.$$
Letting $a$ to $+\infty $, we obtain the claimed inequality for bounded functions. If $\varphi$ is unbounded we apply the 
inequality to $\min\big(|\varphi|,n \big)$ for $n$ growing to infinity and conclude by monotone convergence.

The proof of the second inequality is similar.  It  uses the remarkable relation $M^*(S'(x))=S(x)$.
\end{proof}

\bigskip
Next, we state the transportation inequality for the exponential law with a cost function comparable to $\min (x^2,|x|)$.
It is the analogue of Talagrand's inequality for the symmetric exponential law \cite{Tal}.
\begin{proposition}\label{prop:t1} Let $\alpha\in (0,1)$ and $c_\alpha(x)=\frac{1-\alpha}{\alpha} \Big(\alpha x-1 +\exp(-\alpha x)\Big)$.
   Let $g\cdot \mu$ be a probability measure. Then
    $$ \ent_\mu(g) \ge T_{c_\alpha}(g\, d\mu,\mu).$$
\end{proposition}
\begin{proof}
   Let $T(x)=x+\theta(x)$ be the non-decreasing map transporting $\mu$ to $g\, d\mu$. Lemma~\ref{lem:it1} with $f=1$  gives
$$ \ent_\mu(g)-\theta(0)= \int \big( \theta'-\log(1+\theta')\big) \, d\mu.$$
The term $\theta(0)$ is the displacement that is applied to the origin. It 
corresponds to the first point of the support of $g \,d\mu$. One way to get rid of this term is to approximate $g_, d\mu$
by a measure with support starting at $0$. Another way is to translate $g$: let $a$ be the first point of the support of 
$g\, d\mu$, then let $g_a(x)=e^{-a}g(x+a)$. It is easy to check that $g_a\, d\mu$ is a probability measure, that $T-a$
pushes forward $\mu$ to $g_a d\mu$ and $\ent_\mu(g_a)= \ent_\mu(g)-a$. 
So without loss of generality, we can assume that $\theta(0)=0$ and by  Lemma~\ref{lem:bh}
 $$ \ent_\mu(g)= \int \big( \theta'-\log(1+\theta')\big) \, d\mu= \int M(\theta') \, d\mu 
\ge \frac{1-\alpha}{\alpha} \int S(\alpha \theta)\, d\mu.  $$
\end{proof}

In order to recover the modified log-Sobolev inequality for $\mu$, we need the following lemma:
\begin{lemma}\label{lem:tc}
    Let $d\mu(x)=e^{-x} \mathbf1_{x>0} \, dx$ be the exponential law and let $d\nu= e^g \, d\mu$  be a probability 
measure. Assume that $g$ is locally Lipschitz and satisfies $|g'|<c$ a.e. for some constant $c<1$.
Then  the monotone map $T$ which transports $\nu$ to $\mu$ verifies 
  $$ T'(x) \in [1-c,1+c], \quad\mbox{for all } x\ge 0.$$
The reciprocal map $S=T^{-1}$ transports $\mu$ to $\nu$ and satisfies $S'(x)\le \frac{1}{1-c}$ for $x\ge 0$.
\end{lemma}

\begin{proof}
   The map $T$ is actually given by 
   $$ T(x)=-\log \left( \int_{x}^\infty  e^{g(u)-u} du\right).$$
   Indeed, this expression is strictly increasing and satisfies for $x\ge 0$
   $$ \mu\big((-\infty,T(x) ]\big)=\int_{0}^{T(x)} e^{-u} du= 1-e^{-T(x)}=\int_{0}^x e^{g(u)-u}du=\nu\big((-\infty ,x]\big).$$ 
   Hence 
   $$T'(x)= \frac{e^{g(x)-x}}{\int_{x}^\infty  e^{g(u)-u} du} 
=  \frac{\int_{x}^\infty \big(1-g'(u) \big)e^{g(u)-u}du}{\int_{x}^\infty  e^{g(u)-u} du} \in [1-c,1+c].$$
  Since $ 1-c >0$, it follows that its reciprocal bijection $S=T^{-1}$
  satisfies $ 0< S'(x)\le \frac{1}{1-c}$.
\end{proof}
\begin{remark}
The above statement is an  elementary companion to 
Caffarelli's celebrated theorem \cite{caff00mpot}: if $\gamma$ is a Gaussian measure on $\mathbb R^d$ and
and $d\mu =e^{-W} d\gamma$ with $W''\ge 0$ then $\mu$ is the image of $\gamma$ by a contraction.
The following heuristic argument allows to understand better the similarities. We work in dimension 1
with two probability measures $d\mu=e^{-V(x)}dx$ and $d\sigma=e^{-W}d\mu$. The monotone transport $S$
from $\mu$ to $\sigma$ satisfies $e^{-V(x)}=e^{-W\big(S(x)\big)-V\big(S(x)\big)} S'(x)$.
If $S$ is smooth enough, taking logarithms and differentiating  gives 
\begin{equation}\label{eq:d1}
 S'(x) V'\big(S(x)\big)+S'(x) W'\big(S(x)\big)=V'(x)+\frac{ S''(x)}{S'(x)}.
\end{equation}
Following  Caffarelli, we assume that $S'$ achieves its maximum at an interior point $x_0$.
Then $S''(x_0)=0$ and the latter equality yields $S'(x_0) \Big(V'\big(S(x_0)\big)+W'\big(S(x_0)\big)\Big)=
V'(x_0)$. For the exponential law, $V'=1$ on the image of $S$. If $W'\ge -c$, the function $S'$ satisfies  at its maximum 
$$ S'(x_0)=\frac{1}{1+W'\big(S(x_0)\big)}\le \frac{1}{1-c}.$$

In the case  $V(x)=x^2/2$, it is natural to differentiate  \eqref{eq:d1} in order to get constant terms $V''=1$ 
$$ S''(x) V'\big(S(x)\big)+S'(x)^2 V''\big(S(x)\big)+S''(x) W'\big(S(x)\big)+S'(x)^2 W''\big(S(x)\big)
=V''(x)+\frac{ S'''(x)}{S'(x)}-\frac{S''(x)^2}{S'(x)^2} \cdot$$
At any point $x_0$ where $S'$ reaches its maximum, $S''(x_0)=0$ and $S'''(x_0)\le 0$, hence
$$ S'(x_0)^2 \Big(1+W''\big(S(x_0)\big)\Big) \le 1.$$
Finally $W''\ge 0$ implies $S'(x_0)\le 1$ and $S$ is a contraction.
\end{remark}

\begin{proposition}\label{prop:lsm1}
 Let $c\in (0,1)$ and  $f:[0,+\infty ) \to (0,+\infty )$ such that $|f'/f|<c$, then
  $$ \ent_\mu(f) \le \frac{4}{(1-c)^2} \int_\R \frac{{f'}^2}{f} d\mu.$$
\end{proposition}
\begin{proof} By homogeneity we may assume that $\int f \, d\mu=1$.
   Let $T(x)=x+\theta(x)$ be the monotone map pushing forward $f\cdot\mu$ to $\mu$. By Lemma~\ref{lem:tc} 
   we know that   $|\theta'|\le c$. This allows to check the growth conditions needed to apply Lemma~\ref{lem:it1}
   with  $f$ and  $g=1$ and to obtain
  \begin{equation}\label{eq:1} 
   -\int f'\theta \, d\mu  =\ent_\mu(f) + \int \big(\theta'-\log(1+\theta')\big)f \, d\mu 
         \ge \ent_\mu(f) + \frac14 \int  (\theta')^2f \, d\mu,
  \end{equation}   
 where we have used that for $|x|\le c<1$, $N(x)=x-\log(1+x) \ge x^2/4$.
To conclude the argument we need to get rid of the $\theta'$ term by means of a Sobolev type inequality 
for the measure $f\cdot \mu$. But thanks to Lemma~\ref{lem:tc}, the hypothesis $|f'/f|<c$ guarantees the existence 
 of a $\frac{1}{1-c}$-Lipschitz map $S=T^{-1}$ pushing forward $\mu$ to $f\cdot\mu$. This classically implies that Sobolev 
type inequalities enjoyed by $\mu$ transfer to $f\cdot\mu$. Indeed assume that  every smooth function $\varphi$ with $\varphi(0)=0$
satisfies $\int N_1(\varphi) \, d\mu\le \int N_2(\varphi') \, d\mu$ where $N_2$ is non-increasing on $\mathbb R^-$ and non-decreasing
 on $\mathbb R^+$. Then  for any smooth $\psi$ vanishing at 0,  and since $S(0)=0$, we may apply the 
  Sobolev inequality to  $\varphi=\psi\circ S$. Since the law of $S$ under $\mu$ is $f\cdot\mu$, we obtain
  \begin{eqnarray*} 
  \int N_1(\psi) \, f\, d\mu&= &\int N_1(\psi \circ S) \,d\mu \le \int N_2 (S' \psi'\circ S) \, d\mu \\
     & \le&    \int N_2 \Big( \frac{1}{1-c} \psi'\circ S\Big) \, d\mu= \int N_2 \Big(\frac{1}{1-c}  \psi'\Big)\, f \, d\mu.
   \end{eqnarray*}
By Lemma~\ref{lem:bh}, we recover the classical Poincar\'e inequality for the exponential law: if $\varphi(0)=0$
then $\int \varphi^2 d\mu \le 4  \int (\varphi')^2 d\mu$ from which we deduce
 $$ \int \theta^2 f\, d\mu \le \frac{4}{(1-c)^2} \int (\theta')^2 f\, d\mu.$$
Plugging this inequality in the above entropy estimate yields
\begin{eqnarray*}
   \ent_\mu(f) &\le & \int |f'|\, |\theta|\, d\mu- \frac{(1-c)^2}{16} \int  \theta^2 f\,  d\mu\\
     &\le & \int f \sup_u \left\{\frac{|f'|}{f}u- \frac{(1-c)^2}{16} u^2 \right\} 
     = \frac{4}{(1-c)^2}\int \frac{(f')^2}{f} \, d\mu.
\end{eqnarray*} 
\end{proof}

\begin{remark}
  The interest of the above proof lies in the interpretation of the condition $|f'/f|<c$ in terms 
  of transport. It does not provide very good constants. Bobkov and Ledoux obtain a constant of the
   form $\frac{2}{1-c}$ which captures the right order in $c$ as one can check with the function $f(t)=(1-c)e^{ct}$.
\end{remark}

\subsection{Inequalities for the Laplace distribution}
Let  $d\nu_1(t)= e^{-|t|}\, dt/2$, $t\in \mathbb R$ be the symmetric exponential law.
We use the relation $\nu_1=\mu * \check{\mu}$, where  $d\check{\mu}(t)=e^t \mathbf1_{t<0}\, dt$
is the image of the exponential law by a reflection. Log-Sobolev and transportation
cost inequalities easily pass to product measures. Hence the previous results on $\mu$ transfer to $\mu\otimes \mu$, and to 
$\nu_1$ by considering functions of $(x,y)$ depending only of $x-y$:  

\begin{proposition}
For $\alpha\in (0,1)$, and any probability measure of the form $f\cdot\nu_1$,
$$ \ent_{\nu_1}(f) \ge 2\frac{1-\alpha}{\alpha}  \inf_{\pi\in \Pi(\nu_1, f\cdot\nu_1)} 
 \int_{\mathbb R^2} \log\cosh\left(\frac{\alpha}{2}(x-y) \right) \, d\pi(x,y).$$

Let $c\in (0,1)$, $f:\mathbb R \to \mathbb R^+$ be a smooth function with $|f'/f|<c$ then
$$\ent_{\nu_1} f \le \frac{8}{(1-c)^2} \int_\R \frac{{f'}^2}{f} \, d\nu_1.$$
\end{proposition}
\begin{remark}
   Talagrand actually proved a slightly stronger transportation inequality, but his proof is a lot more involved.
   Bobkov and Ledoux also had a better constant in the log-Sobolev inequality.
\end{remark}

\subsection{Gaussian transportation cost inequality for  measures with median at 0}
Applying Lemma~\ref{lem:it1} to $\mu=\gamma$ the standard Gaussian 
measure on $\mathbb R$, $f=1$ amounts to reproducing Talagrand's proof of
the Gaussian transportation cost inequality \cite{Tal}:
$$ \ent_\gamma(g) = \int \frac{\theta^2}{2} d\gamma + \int \big( \theta'-\log(1+\theta')\big) \, d\gamma
   \ge  \int \frac{\theta^2}{2} d\gamma = \frac{1}{2} T_2(\gamma,g\cdot\gamma).$$
This inequality is known to be stronger than the Poincar\'e inequality for $\gamma$, and strictly weaker than
the Gaussian logarithmic Sobolev inequality. A natural question  asks for improvements of these
inequalities for even functions, or  for centered functions (i.e.  $\int x f(x) \, d\gamma(x)=0$).
It is known that the  Poincar\'e constant   may be improved by a factor 2 for centered functions, 
whereas the log-Sobolev inequality does not improve for even functions. It was recently understood that the $\tau$-property
can be improved for centered functions \cite{maur91sdi,artskm04spff}.
The constant $\frac12$ in the  above transportation cost inequality cannot be improved for symmetric measures, as shown by 
 the following example.
Consider for $a\ge 0$ the probability measure $m_a$ defined by
   $$ dm_a(x)=(2\pi)^{-1/2}\left( e^{-\frac{(x-a)^2}{2}} \1_{x<0}+ e^{-\frac{(x+a)^2}{2}} \1_{x>0}\right)\, dx,
  \qquad x \in \mathbb R.$$
Clearly the map $T$ defined by $T(x)=x-a$ for $x\le 0$ and $T(x)=x+a$ for $x>0$ pushes $\gamma$ forward to $m_a$.
It is monotone and therefore optimal form the quadratic cost. Since for all $x$, $|T(x)-x|=a$, it follows
that $T_2(\gamma,m_a)=a^2$.
 Let $g_a =\frac{dm_a}{d\gamma}$. By a straightforward calculation
 $$ \ent_\gamma (g_a)=\frac{a^2}{2}+ \sqrt{\frac{2}{\pi}} a \sim_{a\to +\infty}\frac{a^2}{2}. $$
 However there is room for an improvement of lower order. 
 \begin{proposition} Let $N(t)=|t|-\log\big(1+|t|\big)$. For all probability measures $g \cdot \gamma$  on $\mathbb R$ with median at 0,
     $$ \ent_\gamma(g)\ge T_w(m,\gamma),$$
 where  the cost function is $w(x)=x^2/2+ N\big( x/\sqrt{2\pi}\big)$. For large $x$, $w(x)= 
    \frac{x^2}{2}+\frac{x}{\sqrt{2\pi}}+o(x)$.
 \end{proposition}
\begin{proof}
Assume that $g$ is positive and continuous. Since the  median of $g\cdot \gamma$ is zero, the 
monotone transport $T$ from $\gamma$ to $g\cdot\gamma$ satisfies $T(0)=0$, hence the displacement vanishes at the
origin: $\theta(0)=0$. Recall 
$$ \ent_\gamma(g) \ge \int \frac{\theta^2}{2} d\gamma + \int \big( \theta'-\log(1+\theta')\big) \, d\gamma
    = \frac{1}{2} T_2(\gamma,g\, d\gamma)+\int \big( \theta'-\log(1+\theta')\big) \, d\gamma.$$
We need a Sobolev type inequality to lower bound the latter term.
Given a smooth function $\varphi$ on $\mathbb R$ vanishing at 0, we apply  Lemma~\ref{lem:bh} to $\varphi(x)$ and  $\varphi(-x)$ and obtain $$\int N(\varphi) \, d\nu_1\le \int N(2\varphi')\, d\nu_1.
$$
One can check that the monotone transport $S$ from $\nu_1$ to $\gamma$ is   $\sqrt{\frac{\pi}{2}}$-Lipschitz. Reasoning as in the proof of Proposition~\ref{prop:lsm1}, the latter inequality  applied to $\varphi=\theta
\circ S$, which vanishes at 0, yields
  $$\int \big( \theta'-\log(1+\theta')\big) \, d\gamma \ge \int N(\theta')\, d\gamma \ge \int N(\theta/\sqrt{2\pi})   \,d\gamma.
 $$
\end{proof}




\section{Generalizations to Riemannian manifolds}

Some results of this paper can be obtained in the Riemannian
setting. This is illustrated  in this section. Several lemmata obviously
extend since they do not use the geometric structure of the space;
we shall use them in the Riemannian setting   without further explanation.

As in the flat case, the starting point here is the
above-tangent lemma.
 Let $(M,g)$ be a smooth, complete, connected
Riemannian manifold without boundary. The geodesic distance on $M$
is denoted by $\rho$ and the Riemannian volume by $v$. The following theorem is an adapted version
of the result from \cite{CEMCS0}, \cite{CEMCS}. The proof is almost the same as
the original one and we omit it here.

\begin{theorem}
\label{CEMCS-th}
Let $d\mu(x) = e^{-V(x)} \,d v(x)$ be a probability measure on $M$,
$g$ and $h$ two compactly supported non-negative functions such that $g\cdot \mu$
and $h \cdot \mu$ are probability measures. Let
 $T(x) = \exp_x (\nabla \theta(x))$ be the optimal transport minimizing
the quadratic transportation cost and pushing forward $g\cdot\mu$ to $h\cdot\mu$.
Then it holds
$$
\mbox{\rm Ent}_{\mu} g \le \mbox{\rm Ent}_{\mu} h - \int_{M} \bigl< \nabla \theta, \nabla g \bigr> \,d\mu
+ \int_{M} \mathcal{D}_V(x,T(x)) g\,d\mu,
$$
where
$$
\mathcal{D}_V(x,T(x))
=
V(x)  + \bigl<\nabla \theta(x), \nabla V(x) \bigr> - V(T(x))
- \int_{0}^{1} (1-t) \mbox{\rm Ric}_{\gamma(t)} \bigl( \dot{\gamma}(t),\dot{\gamma}(t)\bigr)\,dt,
$$
where $\gamma$ is the geodesic joining $x$ and $T(x)$ given by $\gamma(t) = \exp(t\nabla \theta(x))$.
\end{theorem}
If $V$ is twice continuously differentiable, one has
 $$
\mathcal{D}_V(x,T(x))
=
-\int_{0}^{1} (1-t) \Bigl( \mbox{\rm Hess}_{\gamma(t)}V  + \mbox{\rm Ric}_{\gamma(t)} \Bigr)
 \bigl( \dot{\gamma}(t),\dot{\gamma}(t)\bigr) \,dt.
 $$

 The next statement is obtained as an application. It nicely complements Wang's theorem:

 \begin{theorem}
\label{BE-bound}
Let $d\mu(x) = e^{-V(x)} \,d v(x)$ be a probability measure on $M$, with a twice continuously differentiable
potential $V$. Let $\alpha>1$ and suppose that there exists $x_0\in M$ and $\varepsilon>0$ such that 
$$
 \exp\big(\varepsilon \rho(x_0,x)^\alpha\big) \in L^1(\mu).
 $$
Assume that  one of the following two conditions is satisfied

$(i)$  $\alpha\in(1,2]$ and pointwize
 $
 { \rm Hess} V + \mbox{\rm Ric} \ge 0
 $

$(ii)$  $\alpha>2$ and   there exists $K \in \R$ such that pointwize
 $
 { \rm Hess} V + \mbox{\rm Ric} \ge K\, \mbox{\rm Id}.
 $

\noindent
 Then there exists $\kappa>0$ such that $\mu$ satisfies the isoperimetric inequality 
$$ \mathcal I_\mu(t) \ge \kappa \min(t,1-t) \log^{1-\frac{1}{\alpha}}\left(\frac{1}{\min(t,1-t)}\right), \quad t\in (0,1).$$
\end{theorem}
\begin{remark}
   By Corollary~\ref{coro:isoalpha}, various functional inequalities follow. Also note that  the results of Wang 
\cite{wang01lsic} and Ledoux~\cite{Led94} provide the case $\alpha=2$: the isoperimetric inequality is valid provided
 $ { \rm Hess} V + \mbox{\rm Ric} \ge K\, \mbox{\rm Id}$ and $\exp\big( (\varepsilon+|K|/2) \rho(x,x_0)^2\big) \in L^1(\mu)$.
Unfortunately our method does not reach $\alpha=1$.
\end{remark}
\begin{proof}
First we assume  $(i)$. Let $\tau=2/\alpha^*$.
We establish Inequality  $I(\tau)$ for $\mu$ along  the same lines as in $\R^d$: We apply Theorem \ref{CEMCS-th}
and note that $\mathcal D_V\le 0$. Hence we only have to deal with the linear term.
 This can be done exactly as in Proposition~\ref{tang-t-est}.

 By Proposition~\ref{prop:local}, $\mu$ satisfies a local 
Poincar\'e inequality.
 By  Theorem \ref{FI}, Inequality $I(\tau)$  implies the corresponding tight
 $F_\tau$-Sobolev inequality. As an intermediate step
 of this argument, it has been established that $\mu$ satisfies a Poincar\'e inequality.
 An argument of Ledoux \cite{Led94} shows that when $
 { \rm Hess} V + \mbox{\rm Ric} 
 $
 is uniformly bounded from below, the spectral gap inequality yields an isoperimetric inequality 
 of Cheeger $\mathcal I_\mu(t)\ge c \min(t,1-t)$. In particular, it is enough to prove the claimed
 isoperimetric bound for $\min(t,1-t)$ small. Ledoux's argument has been adapted to other functional
 inequalities: the $F_\tau$ inequality implies an isoperimetric inequality of the form
 $\mathcal I_\mu(t)\ge c' \min(t,1-t) \,F_\tau\big(1/\min(t,1-t)\big)^{1/2}$ when $\min(t,1-t)$ is 
 small. This is explained in Section 4 and 8 of  \cite{bartcr06iibe}; it also follows from different
 arguments of \cite{Wang00}. The proof  is complete under Condition $(i)$.
 
  \smallskip
  For $(ii)$, assume that $K\le 0$.
  Reasoning as in the proof of Corollary~\ref{wangstrong}, we get that $$\int 
  e^{\frac{K+\varepsilon}{2}\rho(x,x_0)^2}
   d\mu(x)<+\infty.$$ Hence Wang's theorem applies and gives in particular that $\mu$ satisfies a Poincar\'e inequality
   (note that the new proof that we gave in the Euclidean case is easily adapted to the Riemannian setting).
   By the result of Ledoux \cite{Led94}, $\mu$ satisfies Cheeger's isoperimetric inequality, and consequently
   it is enough to prove the claimed isoperimetric inequality for small values of $\min(t,1-t)$.
   Our strategy is to prove a (defective) modified LSI with cost $t^\alpha$.
  To do this, we apply the above tangent lemma with $f^2\cdot \mu$ and $\mu$. The linear term is estimated
  as in Lemma \ref{lem:lin1}: since $|\nabla \theta(x)|=\rho(x,T(x))$, for any $\eta >0$
  $$ 2\int \big< \nabla f, -\nabla \theta\big> f\, d\mu \le \frac{\eta}{\alpha^*} \int \left|\frac{ 2 \nabla f}{\eta f}\right|^{\alpha^*}
  f^2 d\mu +\frac{\eta}{\alpha} \int \rho(x,T(x))^\alpha f(x)^2 d\mu(x).$$ 
  On the other hand 
  $$ \mathcal D_V(x,y) \le \frac{K}{2} \rho(x,T(x))^2 \le \eta \rho(x,T(x))^\alpha+ N(K,\alpha,\eta).$$
  Hence we are done if we show that for some $\eta, \delta\in (0,1)$ and $C\in \R$,
  $$ \eta(1+\alpha^{-1}) \int \rho(x,T(x))^\alpha f(x) ^2 d\mu(x) \le (1-\delta) \ent_\mu(f^2) + C. $$
This is done  as in the proof of Corollary~\ref{coro:wang} using 
$\rho(x,T(x))^\alpha \le (1+\eta_1) \rho(x,x_0)^\alpha+ M(\eta_1)  \rho(T(x),x_0)^\alpha$, the duality of entropy  and 
the integrability property. The {\bf(MLSI)} with cost $t^\alpha$, or equivalently the ($\alpha^*$-{\bf LSI}), implies the 
claimed isoperimetric inequality for small values. This is explained in the next lemma.
\end{proof}
The next result extends to $q\in (1,2)$ a statement of Ledoux \cite{Led94} for $q=2$.
\begin{lemma}
  Let $q\in (1,2)$.  Let $\mu= e^{-V(x)}\cdot dv(x)$ be a probability measure on $M$. Assume that there exists $K\in \R$ such that 
 $
 { \rm Hess} V + \mbox{\rm Ric} \ge K\, \mbox{\rm Id}
 $
 and that $\mu$ satisfies a possibly defective $q$-log-Sobolev inequality. Then there exists $t_0\in (0,\frac12], \kappa>0$ such that 
  $$ \mathcal I_\mu(t) \ge \kappa \min(t,1-t) \log^{\frac{1}{q}}\left(\frac{1}{\min(t,1-t)}\right), \quad t\in (0,t_0)\cup (1-t_0,1).$$
\end{lemma}

\begin{proof}
   The {\bf (qLSI)} is equivalent by a change of functions to the following defective {\bf MSLI}
 $$ \ent_\mu(f^2) \le B \int f^2 \left|\frac{\nabla f}{f}\right|^q d\mu + D \int f^2 d\mu.$$
Since $2/q>1$, Young's inequality yields for $t\ge 0,\eta>0$ that $t^{q/2} \le \frac{q}{2} \eta t +\frac{2-q}{2}
   \eta^{\frac{-q}{2-q}}$. For $t=|\nabla f/f|^2$, $\eta=2\varepsilon/q$  we get for some $m> 0$ and all $\varepsilon>0$
 $$  \ent_\mu(f^2) \le \varepsilon \int |\nabla f|^2 d\mu + \Big( D+m \varepsilon^{\frac{-q}{2-q}}\Big) \int f^2 d\mu.$$
Set $\beta(\varepsilon)= D+m \varepsilon^{\frac{-q}{2-q}}$. By a celebrated theorem of Gross, any log-Sobolev inequality satisfied by $\mu$ implies continuity properties of the semigroup $(P_t)$ 
generated by $L=\Delta-\nabla V \cdot \nabla$, see e.g. \cite{Gross-surv}.
 Denoting $\|f\|_p=\big(\int |f|^p d\mu)^{1/p}$, this 
theorem yields for all $\varepsilon, t>0$, and all $f$,
\begin{equation}\label{eq:gross}
\| P_t f\|_2 \le \exp\left(\frac{\beta(\varepsilon)}{2}\cdot                                    \frac{e^{4t/\varepsilon}-1}{e^{4t/\varepsilon}+1}\right) 
\|f\|_{1+e^{-4t/\varepsilon}}.
\end{equation}
On the other hand a theorem of Ledoux \cite{Led94} improved in \cite{BakLed} shows that under 
the condition  ${ \rm Hess} V + \mbox{\rm Ric} \ge K \, \mathrm{Id}$, there exists a constant $C>0$ 
such that for all $t \in (0, 1/|K|)$, and all Borel sets $A\subset M$,
$$ \mu^+(\partial A) \ge \frac{C}{\sqrt{t}} \Big(\mu(A)- \|P_t \mathbf 1_A\|_2^2 \Big).$$
Combining this fact with \eqref{eq:gross} for $f=\mathbf 1_A$ gives for $t < |K|^{-1}$, $\varepsilon>0$
\begin{equation}\label{eq:te}
 \mu^+(\partial A) \ge C \frac{\mu(A)}{\sqrt t} \left[1- \exp\left(\big(\beta(\varepsilon)+\log(\mu(A))\big) \,
 \frac{e^{4t/\varepsilon}-1}{e^{4t/\varepsilon}+1} \right) \right].
\end{equation}
It remains to make a good choice of $\varepsilon,t$. The idea is to fix $\varepsilon$ so that $\beta(\varepsilon) \sim \frac{1}{2}
 \log(1/\mu(A))$ and then to choose $t$ so that $t/\varepsilon\sim 1/\log(1/\mu(A))$, which is small if we consider
sets of small measure. More precisely, we set
 $$ \varepsilon=\left( \frac{1}{2m} \log\frac{1}{\mu(A)}\right)^{\frac{q-2}{q}} \qquad \mbox{and}
\qquad t=\left(\log\frac{1}{\mu(A)}\right)^{-\frac{2}{q}}.$$
When $\mu(A)$ is small enough this is compatible with the constraint $t<|K|^{-1}$. In particular
this choice implies $\beta(\varepsilon)+\log \mu(A)=D-\frac{1}{2} \log \frac{1}{\mu(A)}$ and
$t/\varepsilon=(2m)^{(q-2)/q} \Big(\log\frac{1}{\mu(A)} \Big)^{-1}$. Consequently, the quantity in brackets
in \eqref{eq:te} has a strictly positive limit when $\mu(A)$ tends to zero and there exists $C'>0$
such that 
$$ \mu^+(\partial A) \ge C' \frac{\mu(A)}{\sqrt t}=C' \mu(A) \log^{\frac{1}{q}}\frac{1}{\mu(A)},$$
when $\mu(A)$ is small enough. The same argument gives a  similar bound  for large sets since $
\mu(A)- \|P_t \mathbf 1_A\|_2^2 =\mu(A^c)- \|P_t \mathbf 1_{A^c}\|_2^2$. 
\end{proof}

We will need more detailed description of the optimal transport of measures on manifolds. See
\cite{CEMCS0}, \cite{CEMCS} for details.

Let   $T(x) = \exp (\nabla \theta(x))$ be the quadratic optimal transportation mapping pushing forward $g\cdot\mu$ to $f\cdot\mu$.
Set:  $T_t(x) = \exp (t\nabla \theta(x))$. 
The change of variables formula reads as
\begin{equation}
\label{ch-var-riem}
g = f(T)  J_1, 
\end{equation}
where 
$$
J_t  = \det Y(t) \bigl(H(t) + t \mbox{Hess} \theta \bigr), 
$$
$Y(t) = \mbox{D} \mbox( \mbox{exp}_x)_{t\nabla \theta} $, $H(t) = \frac{1}{2} \mbox{Hess}_x \rho^2(x,T_t(x))$.
Here  $\mbox{Hess} \theta$ is understood in the sence of Alexandrov due to a local semiconvexity of $\theta$.
There exist the following relations between the volume distortion coefficients  $v_t(x,y)$
and $Y(t)$, $H(t)$  (see \cite{CEMCS0}, 
\cite{CEMCS} for the precise definition)
\begin{equation}
\label{distor1}
v_t(x,T(x)) = \det Y(t) Y^{-1}(1),
\end{equation}
\begin{equation}
\label{distor2}
v_{1-t}(T(x),x) = \det \frac{Y(t)\bigl( H(t) - t H(1) \bigr)}{1-t}.
\end{equation}
If $\mbox{Ric} \ge k(d-1) \mbox{Id}$, where $k \le 0$,  the volume distortion coefficients can be estimated by the 
Bishop's comparison theorem 
\begin{equation}
\label{Bishop}
v_t(x,y) \ge
\Bigl( \frac{S_k(t\rho(x,y))}{S_k(\rho(x,y))}\Bigr)^{d-1},
\end{equation}
where
$S_k(t) = \frac{\sinh (k^{1/2} t) }{k^{1/2} t }$.

The following theorem is a generalization of Theorem \ref{tightMLSI-transport2} a) and Theorem \ref{20.04}
(see also Remark \ref{rem:KS}).

\begin{theorem}
\label{LSI-m}
Let $M$ satisfy  $\mbox{\rm Ric} \ge K \mbox{\rm Id}$, $K \le 0$. Assume that
$\mu = e^{-V} \,d v$ is a probability measure with twice continuously differentiable
 potential such that one of the following assumptions is fulfilled
\begin{itemize}
\item[a)]
$$
\exp(\delta \rho(x,x_0)^{\alpha}) \in L^1(\mu)
$$
for some $\delta > 0 $, $1< \alpha \le 2$,
$$
V \le N_1 \bigl( -\Delta  \phi + \bigl<\nabla V, \nabla \phi\bigr> \bigr) + N_2,
$$ 
where $N_1 >0$, $\phi = \rho^2(x,x_0)$ for some $x_0 \in M$. In addition, assume that
$$
\exp \bigl( \varepsilon  |\nabla V| \log^{\frac{1}{\alpha}}|\nabla V| \bigr) \in L^1(\mu)
$$
and $-V \le g$ with $g \in L^1(\mu)$
\item[b)]
$$
\exp(\delta \rho(x,x_0)^{\alpha}) \in L^1(\mu)
$$
for some $\delta > 0 $, $1< \alpha \le 2$, $x_0 \in M$ and
$V$ is bounded from below and for $s>0$, $0<t<1$
$$
s \max(1,V)^{\tau} \le (1-t) | \nabla V|^2 -\Delta V + C,
$$
where $\tau = \frac{2}{\beta}$. 
\end{itemize}

Then  $\mu$ satisfies 
inequality $(I_\tau)$,  $F_\tau$ -inequality and modified Sobolev inequality with $c=c_{\alpha}$.
\end{theorem}
\begin{proof}
Exactly as in the previous theorem it is sufficient to prove the defective $(I_\tau)$-inequality.
In  order not to repeat lengthy arguments, we prove only the case $\alpha=2$ in b). The 
proofs of the case $\alpha \ne 2$ and the item a) can be obtained from the proofs of
Theorem \ref{20.04} and Theorem   \ref{tightMLSI-transport2} respectively by the similar modifications.

Set:  $r(x)=\rho(x_0,x)$.
First we note that by a comparison theorem  the volume of the ball $\{x: \rho(x,x_0) \le r \}$ grows 
mostly exponential as a function of $r$. Hence  $\exp(-t r^2) \,d v$
is a finite measure for every $t>0$. 
 Consider the quadratic transportation $T$ of $f^2\cdot \mu$
to $\mu$, where $f$ is smooth and compactly supported. 

Applying (\ref{ch-var-riem}) and integrating the logarithm of both sides with respect to $\mu$, we get
$$
\int_{M} f^2 \log f^2 d \mu = 
\int_{M} V  f^2 d \mu
-
\int_{M} V  d \mu
+
\int_{M} \log \det \Bigr( H(1) + \mbox{Hess} \theta \Bigr) f^2 d \mu
+
\int_{M} \log \det Y(1)  f^2 d \mu.
$$

The term $\int_{M} V  f^2 d \mu$ can be estimated as in the flat case (see Theorem \ref{tightMLSI-transport2} and Theorem \ref{20.04} ).
Further applying (\ref{distor1}) with $t=0$ and (\ref{Bishop}), we get that
$\log \det Y(1)$  can be estimated by $C\rho(x,T(x))$ for big values of $\rho(x,T(x)$   and by $C\rho^2(x,T(x))$ for small values.
Applying the  triangle inequality 
$$\rho(x,T(x)) \le r +  \rho(T(x),x_0),
$$ the Young inequality and change of variables, it is easy to show that the term
 $\int_{M} \log \det Y(1)  f^2 d \mu$ is dominated by $\varepsilon \mbox{Ent}_{\mu} f^2$ for any small $\varepsilon$.
Further,  by the Jensen inequality
$$
\int_{M} \log \det \Bigr( H(1) + \mbox{Hess} \theta \Bigr) f^2 d \mu
\le
d \log \int_{M} \mbox{Tr} \Bigl( \frac{ \frac{1}{2} \mbox{Hess}_x \rho^2(x,T(x)) + \mbox{Hess} \theta}{d}\Bigr) f^2 d \mu.
$$
Next we note that
$$ 
\frac{\mbox{Tr} \ \mbox{Hess}_x \rho^2(x,T(x))}{2}
=
|\nabla_x \rho(x,T(x)) |^2 
+
\rho(x,T(x)) \Delta_x \rho(x,T(x)).
$$
By a comparison result  $\Delta_x \rho(x,T(x))$ is dominated by $\Delta_H r(o,x)|_{r(o,x)= \rho(x,T(x))}$, where 
$r(o,x)$ is the distance in the model space with constant curvature $K$ from some fixed pont $o$. This implies, in particular, that
$\Delta_x \rho(x,T(x))$ is bounded for big values of  $\rho(x,T(x))$. Since $\rho^2(x,x_0)$ is smooth in the neighborhood of $x_0$, hence
$$ 
\frac{\mbox{Tr} \ \mbox{Hess}_x \rho^2(x,T(x))}{2}
\le
C\bigl(1
+
\rho(x,T(x)) \bigr)
$$
and
$$ 
\int_M \frac{\mbox{Tr} \ \mbox{Hess}_x \rho^2(x,T(x))}{2} f^2 d \mu
\le
C\bigl(1
+
\int_M \rho(x,T(x)) f^2 d \mu \bigr).
$$
In addition,  
$$
 \int_{M} \frac{ \Delta \theta}{d} f^2 d \mu
\le
  -\frac{2}{d} \int_{M}  \bigl< \nabla \theta,  \nabla f \bigr>   f d \mu
$$
Here we estimate the Alexandrov's Laplacian by the distributional one from above. This is possible since
$\theta$ is locally semi-convex and the singular part of its Laplacian is non-negative. Then we apply integration by parts.
Using
$\bigl| \nabla \theta \bigr| = \rho(x,T(x)),
$
we arrive at the following estimate
$$ 
\Bigl| \int_{M}  \bigl< \nabla \theta,  \nabla f \bigr>   f d \mu \Bigr|
\le
2 \int_{M} \rho^2(x,T(x)) f^2   d \mu + 2 \int_{M} |\nabla f|^2 d \mu. 
$$
 The rest proceeds in the standard way. This means that we apply the triangle inequality
$\rho(x,T(x)) \le r +  \rho(T(x),x_0)$ and make the change of variables for $\rho(T(x),x_0)$.
Then we  estimate $\log x$ by $\varepsilon x + N(\varepsilon)$, apply the standard Young inequality  
and choose a sufficiently small  small $\varepsilon$. This completes the proof.
\end{proof}

\begin{corollary}
\label{riem-power}
Let $M$ satisfy  $\mbox{\rm Ric} \ge K \mbox{\rm Id}$, $K \in \R$ and
$$
V = \delta \rho^{\alpha}(x_0,x) + N
$$
for some $x_0 \in M$, $1< \alpha \le 2$, $\delta>0$, $N \in \R$.
Then  $\mu$ satisfies 
inequality $(I_\tau)$,  $F_\tau$-inequality and modified log-Sobolev inequality with $c=c_{\alpha}$,
where  $\tau = \frac{2}{\beta}$.
\end{corollary}
  \begin{proof}
  As we know from the previous proof
$$
  \Delta r^2 \le C'(1+ r)
  $$
  for some $C'$ in points of differentiability of $r^2$. In addition,
  $|\nabla r|=1$ almost everywhere.
  Function $r^2$ is differentiable outside of $C_{x_0}$, where $C_{x_0}$ is the cut-locus of $x_0$.
  It is known that $C_{x_0}$ has measure zero. Formally applying    Theorem  \ref{LSI-m} b), we get the result.
Nevertheless, since $V$ is not smooth everywhere, the proof needs some justification. 
Analyzing the proof of Theorem \ref{LSI-m}, we see that the estimate
$$
\int_M r^2 f^2 d\mu \le C_1 + C_2  \int_M r f^2 d\mu +  C_3  \int_M r |\nabla f|  f d\mu
$$
should be justified. This can be done by integration by parts formula with the help of Calabi lemma
(see, for instance, \cite{bakrq05vctj}) :
there exists an increasing  sequence of precompact starshaped domains $D_n$ with smooth boundaries which  union is
$M \setminus   C_{x_0}$. In addition, $r^2$ is smooth in every $D_n$ and $\bigl<\nabla r, \nu\bigr> > 0$ on $\partial D_n$  
where $\nu$ is the normal 
outward  vector field on $\partial D_n$.
  \end{proof}

\begin{remark}
Corollary  \ref{riem-power} for the case of $F$-inequalities  has been proved by Wang in \cite{Wang00} for $\alpha >1$.
 It follows from his more general result obtained from a Nash-type inequality by perturbation techniques.
\end{remark}

Finally, we show that the Euclidean logarithmic Sobolev
inequality implies modified log-Sobolev inequalities for special types of measure on manifolds with
the lower Ricci curvature bound.

\begin{theorem}
Assume that $\mbox{Ric} \ge K \mbox{Id}$ and the Riemannian volume
measure satisfies the logarithmic Sobolev inequality in the Euclidean form:
\begin{equation}
\label{ELSI}
\int_{M} f^2 \log \Bigl( \frac{f^2}{\int_{M} f^2 dv} \Bigr) dv \le A \ln \Bigl( B \int_{M} |\nabla f|^2 dv \Bigr),
\end{equation}
where $A$ and $B$ are positive constants. Let $\mu$ be a probability measure of the type 
$$\mu = \frac{1}{Z_{\alpha,N}} \exp(-N \rho^{\alpha}(x,x_0)) dv,$$
where $1< \alpha \le 2$, $N>0$ and $x_0$ is a fixed point. Then  $\mu$ satisfies 
$(I_{\tau})$-inequality with $\tau = 2\bigl(1-\frac{1}{\alpha}\bigr)$.

In addition, $\mu$ satisfies the $F_\tau$ -inequality and modified Sobolev inequality with $c=c_{\alpha}$.
\end{theorem}
\begin{proof}
As above, applying the tightening techniques, it is sufficient to show $(I_{\tau})$-inequality.  Without loss of generality assume that $N=1$. 
In addition, since inequalities of this type are stable under bounded perturbations, it is sufficient
to prove the result for the measure $\nu = \frac{1}{A_{\alpha}}\exp \bigl(-p \bigr) dv$,
where $p = \varphi(\rho(x,x_0))$ and $\varphi(t)$ is a twice continuously differentiable function which is  
equal to $t^{\alpha}$ for $t \ge 1$ and 
quadratic for small values of $t$.
Let $g$ be a smooth function such that  $\int_{M} g^2 d\nu =1$. 
Set: 
$$
f= \frac{1}{A^{1/2}_{\alpha}} \ g \ \exp \Bigl(-\frac{p}{2} \Bigr) .
$$
Obviously, $\int_{M} f^2 dv =1$. Applying (\ref{ELSI}),  one gets
\begin{equation}
\label{04.08}
 \int_M g^2 \Bigl( \ln g^2 - p - \ln A_{\alpha}\Bigl) d\nu
\le
A \ln \Bigl( B \int_M \Bigl[ \nabla g - \frac{\nabla p}{2} g \Bigr]^2 d\nu \Bigr).
\end{equation}
Hence
$$
\int_M g^2 \ln g^2 d\nu \le \int_M g^2 p \ d\nu + \ln A_{\alpha}
+
A \ln \Bigl( B \int_M \Bigl[ \nabla g - \frac{\nabla p}{2} g \Bigr]^2 \Bigl( 1+ \log^{1-\tau} (e+g^2) \Bigr)d\nu \Bigr).
$$
Now we want to apply  integrations by parts to the term
$$
-  \int_M \Bigl< \nabla g, \nabla p \Bigr> g \Bigl( 1+ \log^{1-\tau} (e+g^2) \Bigr)d\nu.
$$
This can not be done directly, since the function $p$ is  differentiable only outside of cut locus of $x_0$.
Nevertheless, proceeding as above with the Calabi lemma and taking into account that 
$p$ is increasing function of the distance, we get that 
 the right-hand side of (\ref{04.08}) can be estimated by
\begin{align*}
 &
A \ln \Bigl( B \int_{M} \Bigl( |\nabla g|^2 - \frac{|\nabla p|^2}{4} g^2 
+ \frac{\Delta p}{2} g^2\Bigr) \bigl(1+ \log^{1-\tau} ( e + g^2) \bigr) d\nu 
+ B \int_{M} g^2 \psi(g^2) g \bigl<\nabla g,\nabla p\bigr> d\nu 
\Bigr)
\end{align*}
where $$\psi = \frac{d}{dx} \Bigl[1+ \log^{1-\tau}(e+x) \Bigr].$$
Further we note that $x\psi(x)$ is bounded and since the Ricci curvature is bounded from below, one has 
$$\Delta p \le C_1 + C_2 r^{\alpha-1}
$$
outside of cut locus, where $r=\rho(x,x_0)$.
Since $|\nabla p|^2 \sim (\alpha-1)^2 r^{2(\alpha-1)}$, $|\nabla p|^2 $ dominates $\Delta p$ for big values of $r$. Applying Cauchy inequality one easily gets that
the right-hand side of (\ref{04.08}) can be estimated by
$$
A \ln \int_{M} \Bigl( N_1 |\nabla g|^2  -\frac{(\alpha-1)^2}{5} r^{2(\alpha-1)} g^2 +  N_2 g^2 \Bigr) \bigl(1+ \log^{1-\tau} ( e + g^2) \bigr) d\nu
$$
for  sufficiently big $N_1, N_2$.

Applying estimate $\ln x \le C x +D(C)$ for a sufficiently big $C$, one finally obtains
$$
\int_M g^2 \ln g^2 d\nu \le \int_M g^2 p \ d\nu + \int_{M} \Bigl( C_1 |\nabla g|^2  -
C_2 \frac{(\alpha-1)^2}{5} r^{2(\alpha-1)} g^2 +  C_3  g^2 \Bigr) \bigl(1+ \log^{1-\tau} ( e + g^2) \bigr) d\nu + C_4,
$$
where $C_2$ can be chosen arbitrary big.
It remains to note that $p \sim r^{\alpha-1}$ for big r, hence by the Young inequality $\int_M g^2 p \ d\nu $
can be estimated by 
$$
 \int \Bigl[ N(\varepsilon) g^2\bigl(1+ \log^{1-\tau} ( e + g^2) \bigr)  +  e^{\varepsilon r^{\frac{2-\alpha}{1-\tau}} } 
+1 \Bigr]r^{2(\alpha-1)} d\nu.
$$
Note that  $\frac{2-\alpha}{1-\tau} = \alpha$, hence
 $e^{\varepsilon r^{\frac{2-\alpha}{1-\tau}} } r^{2(\alpha-1)} \in L^1(\nu)$. Finally, choosing a  sufficiently big $C_2$, we  make the 
term  $\int_M g^2 p \ d\nu$ disappear.
Since
$ \int_{M} g^2  \bigl(1+ \log^{1-\tau} ( e + g^2) \bigr) d\nu $ is dominated by $\varepsilon \mbox{Ent}_{\nu}g^2$ for any positive $\varepsilon$,
we immediately  get the desired estimate.
\end{proof}

\begin{remark}
It is known that inequality (\ref{ELSI}) holds for the hyperbolic space $H^d$ (see \cite{Beck2}). Thus,  we get
another proof of a partial case of Corollary \ref{riem-power}.
\end{remark}


\bigskip
\noindent
F. Barthe: Institut de Math\'ematiques de Toulouse,
  CNRS UMR 5219. Universit\'e Paul Sabatier. 31062 Toulouse cedex 09. FRANCE.
 Email: barthe@math.ups-tlse.fr

\bigskip
\noindent
A. Kolesnikov: Moscow State University of Printing Arts, 2A Pryanishnikova, 127550 Moscaow. RUSSIA.
 Email: sascha77@mail.ru
\end{document}